\documentclass[a4paper]{amsart}
\usepackage{amsmath,amsthm,amssymb,latexsym,epic,bbm,comment,yfonts,mathrsfs}
\usepackage{graphicx,enumerate,stmaryrd,rotating}
\usepackage[all]{xy}
\xyoption{2cell}

\allowdisplaybreaks
\usepackage[active]{srcltx}
\usepackage[parfill]{parskip}

\usepackage{youngtab}

\usepackage[latin1]{inputenc}
\usepackage{amsxtra}
\usepackage{amsmath}
\usepackage{amscd}
\usepackage{amssymb}
\usepackage{amsfonts}
\usepackage[all]{xy}
\usepackage{mathrsfs}
\usepackage{amsthm}
\usepackage{color}
\usepackage{upgreek}
\usepackage{verbatim}
\usepackage{enumerate}
\usepackage{latexsym}
\usepackage{graphicx}
\usepackage{tikz}
\usepackage{todonotes}
\usepackage{setspace}
\usepackage{hyperref}
\usepackage{stmaryrd}
\usepackage{nicefrac}
\usepackage{euscript}
\usetikzlibrary{decorations.pathmorphing}

 \definecolor{darkgreen}{HTML}{336633}
 \definecolor{darkred}{HTML}{993333}
\makeatletter

\newcommand{\arxiv}[1]{\href{http://arxiv.org/abs/#1}{\tt
    arXiv:\nolinkurl{#1}}}

\theoremstyle{plain}
\newtheorem{thm}{Theorem}

\newtheorem{lem}[thm]{Lemma}
\newtheorem{prop}[thm]{Proposition}

\newtheorem{cor}[thm]{Corollary}
\newtheorem{df-prop}[thm]{Definition-Proposition}

\theoremstyle{definition}

\theoremstyle{remark}
\newtheorem{rem}[thm]{Remark}
\newtheorem{ex}[thm]{Example}

\def\onto{\twoheadrightarrow}

\def\Res{\operatorname{Res}\nolimits}
\def\Ind{\operatorname{Ind}\nolimits}
\def\Ext{\operatorname{Ext}\limits}

\def\id{\operatorname{id}\nolimits}

\def\ep{\epsilon}
\def\gl{\mathfrak{gl}}

\def\la{\lambda}

\def\pn{\mf{pe} (n)}
\def\ov{\overline}

\newcommand{\mc}{\mathcal}
\newcommand{\mf}{\mathfrak}
\newcommand{\C}{\mathbb C}
\newcommand{\oo}{{\ov 0}}
\newcommand{\oa}{{\bar 0}}
\newcommand{\ob}{{\bar 1}}
\newcommand{\vare}{\epsilon} 

\newcommand{\fg}{\mathfrak{g}}

\newcommand{\fb}{\mathfrak{b}}

\newcommand{\fn}{\mathfrak{n}}

\newcommand{\mZ}{\mathbb{Z}}

\newcommand{\mC}{\mathbb{C}}

\newcommand{\h}{\mathfrak{h}}

\newcommand{\Coind}{{\rm Coind}}

\newcommand{\g}{\mathfrak{g}}

\newcommand{\fp}{\mathfrak{p}}

\newcommand{\nb}{{\nabla}}
\newcommand{\Z}{{\mathbb Z}}

\def\epd{{\emph{pd}}}
\def\pd{{\text{pd}}}
\def\eid{{\emph{id}}}
\def\id{{\text{id}}}

\def\gld{{\text{gl.dim}}}

 \def\uv{\underline}

\begin{document}

\numberwithin{equation}{section}

\title[Some homological properties of
category $\mc O$ for Lie superalgebras]{Some homological properties of
	category $\mc O$\\ for Lie superalgebras}

\author{Chih-Whi Chen  and Volodymyr Mazorchuk}
\date{}

\begin{abstract}  
For classical Lie superalgebras of type I, we provide necessary and sufficient 
conditions for a Verma supermodule $\Delta(\la)$ to be such that every non-zero 
homomorphism from another Verma supermodule to $\Delta(\la)$ is injective. This 
is applied to describe the socle of the cokernel of an inclusion of Verma supermodules
over the periplectic Lie superalgebras $\pn$ and, furthermore, to reduce the problem of
description of $\mathrm{Ext}^1_{\mc O}(L(\mu),\Delta(\la))$ for $\pn$  to the similar 
problem for the Lie algebra $\mathfrak{gl}(n)$. 

Additionally, we study the projective and injective dimensions of structural 
supermodules in parabolic category $\mc O^\fp$ for classical Lie superalgebras. 
In particular, we completely determine these dimensions for structural supermodules 
over the periplectic Lie superalgebra $\pn$ and the ortho-symplectic Lie 
superalgebra $\mf{osp}(2|2n)$. 
\end{abstract}

\maketitle


\vspace{5mm}

\section{Introduction}\label{sec1}

Contempoary representation theory uses a wide range of homological
methods. This makes understanding homological properties of 
algebraic structures an interesting and important problem.

In Lie theory, one of the most important representation theoretic
objects of study is the Bernstein-Gelfand-Gelfand (BGG)
category $\mathcal{O}$ associated to a fixed triangular 
decomposition of a semisimple finite dimensional complex Lie 
algebra, introduced in \cite{BGG0,BGG}. This category has a number of
remarkable properties and applications, see \cite{Hu08} and
references therein.
Study of various homological invariants of category $\mathcal{O}$ 
goes back to the original paper \cite{BGG} which determines its 
global dimension. Some time ago, the second author started a
series of papers \cite{Ma1,Ma2} aimed at a systematic
study of homological invariants of various classes of structural
modules in category $\mathcal{O}$. This series was 
subsequently continued with \cite{CM1,CoM2,KMM1} covering a wider
class of questions and objects and also extending the results
to Rocha-Caridi's parabolic generalization of $\mathcal{O}$
from \cite{RC}.

Both category $\mathcal{O}$ and its parabolic generalization, 
can be defined, in a natural way, in the setup of Lie superalgebras,
see \cite{Ka,Mu12,ChWa12,Ma,AM}. The objects obtained in this way
are very interesting and highly non-trivial, see \cite{Mu12,ChWa12}
and references therein. Therefore it is natural to study homological
invariants for category $\mathcal{O}$ for Lie superalgebras.
The first step here would be to find out which of the classical
results (i.e. results for Lie algebras) can be generalized to the
super-setting and how. This is one of the motivations behind
some parts of the present paper.

The recent paper \cite{KKM} discovered an interesting connection
between bigrassmannian permutations and cokernels of inclusions
of Verma modules for the Lie algebra $\mathfrak{sl}_n$. 
The second motivation of this paper is to investigate to which
extent this result generalizes to the super-setting. This is 
a non-trivial challenge as it is very well-known that, unlike
the classical case, homomorphisms between Verma supermodules over
Lie superalgebras are often not injective. 

To our surpize, it turns out that, both for the homological 
properties of structural supermodules and for the description of 
cokernels of inclusions of Verma supermodules, it is possible to 
obtain fairly complete results for the periplectic
Lie superalgebra $\pn$. This is the main aim of the
present paper.

Let us now briefly describe the structure of the
paper and the main results. In Section~\ref{Sect::pre} 
we collect preliminaries on classical Lie superalgebras 
which are relevant for the rest of the paper. 

In Section~\ref{Sect::HomVerSuper} we describe Verma supermodules 
for classical Lie superalgebras of type I having the property that
any homomorphism from any other Verma supermodule to
them is injective. As an application, we prove that every 
non-zero homomorphism between Verma supermodules for $\pn$ is an 
embedding.
 
In Section~\ref{Sect::soccokerpn}, we describe the socle of the 
cokernel of a nonzero homomorphism between Verma supermodules 
for  $\pn$. We manage to reduce the problem for $\pn$ to the 
corresponding problem for $\gl(n)$, where the main result of
\cite{KKM} applies.

Section~\ref{Sect::homodim} is devoted to the study of possible connections 
between homological invariants for a Lie superalgebra
and the corresponding homological invariants for its even
part, which is a Lie algebra. In particular, we relate
the finitistic dimension of the category of
supermodules for a classical Lie superalgebra to the 
global dimension of the category of modules for its 
even part.  In Sections~\ref{sect::52} and \ref{sect::54} we 
focus on projective and injective dimensions of structural 
modules in an arbitrary parabolic cateogry $\mc O^\fp$. 
In Section~\ref{sect::55} we introduce the notion of 
the associated variety and discuss its relation to 
projective dimension.  
 
Finally, in  Section~\ref{Sect::6}  we investigate the 
projective dimensions of structural modules in the 
parabolic category $\mc O^\fp$ for $\pn$  and the ortho-symplectic Lie 
	superalgebra $\mf{osp}(2|2n)$. 
\vspace{3mm} 
 
\noindent
{\bf Acknowledgments.} The first author is partially supported by MoST grant
108-2115-M-008-018-MY2. For the second author the research was 
partially supported by the Swedish Research Council
and G{\"o}ran Gustafssons Stiftelse.  We thank Kevin Coulembier for helpful comments.

\section{Preliminaries} \label{Sect::pre}

Throughout, we let   $\mf g = \mf g_\oo\oplus \mf g_{\ob}$   
be a finite-dimensional classical Lie superalgebra. 
This means that $\mf g_\oa$ is a reductive Lie algebra 
and  $\mf g_\ob$ is a completely reducible $\mf g_\oa$-mo\-du\-le 
under the adjoint action.
  
For a given classical Lie superalgebra, we fix triangular decomposition 
\begin{align} \label{eq::tridec}
& \mf g_\oa =\mf n_\oa \oplus \mf h_\oa \oplus \mf n^-_\oa
\quad\text{ and }\quad
\mf g =\mf n \oplus \mf h \oplus \mf n^-
\end{align}   
of $\mf g_\oa$ and of $\g$, respectively, in the sense of 
\cite[Section~2.4]{Ma}. Here  
$\mf b_\oa:=\mf n_\oa \oplus \mf h_\oa$ and 
$\mf b:=\mf n \oplus \mf h$ are  Borel subalgebras in 
$\mf g_\oa$ and $\mf g$, respectively. In the 
present paper we assume that the Cartan subalgebra 
$\mf h$ is purely  even, i.e., $\mf h=\mf h_\oa$.  
We denote by $\displaystyle \mf g=\bigoplus_{\alpha\in \h^\ast}\mf g^\alpha$
the root space decomposition of $\mf g$ with respect to $\mf h^\ast$.
We denote by $W$ the Weyl group of 
$\mf g_\oa$.  The usual dot-action of $W$ on 
$\mf h^\ast$ is defined as 
$$w\cdot \la := w(\la+\rho_\oa)-\rho_\oa,~\text{for any }w\in W,~\la \in \h^\ast,$$ where $\rho_\oa$ is the half of the sum of all 
positive roots of $\mf g_\oa$. 
For a given weight $\mu \in \h^\ast$, we let $W_\mu$ denote 
the stabilizer of $\mu$ under the dot-action of $W$.    
Let $\ell: W\rightarrow \mathbb{Z}_{\geq 0}$ 
denote the usual length function. 
  
Finally, a weight is called {\em integral},
{\em dominant} or {\em anti-dominant} if it is
integral, dominant or anti-dominant as a $\mf g_\oa$-weight,
respectively. 
  
Finally, we denote by $\Phi^+, \Phi^+_\oa$ and $\Phi^+_\ob$ the sets of positive roots, even positive roots and odd positive roots with respect to the triangular decomposition \eqref{eq::tridec}, respectively. In other words, these are exactly the non-zero 
weights of $\mf b$, $\mf b_\oa$ and $\mf b_\ob$, respectively.
  
\subsection{BGG category $\mc O$}
  
Associated to the triangular decomposition \eqref{eq::tridec} 
above, the BGG category $\mc O  $ (resp. $\mc O_\oa$)
is defined as the full subcategory of the category of
finitely generated  $\g$-modules (resp. $\mf g_\oa$-modules)  
which are semisimple over $\mf h$ and locally $U(\fn^+)$-finite 
(resp. $U(\fn^+_\oa)$-finite), see \cite{BGG, Hu08}.
  
For $\la \in \h^\ast$, the corresponding Verma modules
over $\mf g_\oa$ and $\mf g$ are defined as  
$$\Delta_{\oa}(\lambda)=U(\fg_{\oa})\otimes_{U(\fb_{\oa})}\mC_{\la}\quad\text{ and }\quad\Delta(\la) := U(\mf g)\otimes_{U(\mf b)}   \mC_{\la},$$
respectively. Here $\mC_{\la}$ is the one-dimensional module over
the Borel subalgebra corresponding to $\lambda$.
The unique simple quotients of $\Delta_\oa(\la)$ and 
$\Delta(\la)$ are denoted by $L_\oa(\la)$ and $L(\la)$, respectively. 
The corresponding dual Verma modules $\nabla_\oa(\la)$ and 
$\nabla(\la)$ are similarly defined as coinduced modules,
see \cite[Subsection~3.2]{Hu08} and 
\cite[Definition~3.2]{CCC} for more details.  
We note that $L(\la) \hookrightarrow \nabla(\la).$ Also, we denote by $T_\oa(\la)$ and $T(\la)$ the tilting modules of highest weight $\la$ in $\mc O_\oa$ and $\mc O$, respectively, see \cite[Section 4.3]{Ma} and \cite[Theorem 3.5]{CCC} for more details.
  
Let us denote by $\Res:=\Res_{\mf g_\oa}^\g$,  
$\Ind:=\Ind_{\mf g_\oa}^\g$ and $\Coind:=\Coind_{\mf g_\oa}^\g$ the
corresponding restriction, induction and coinduction functors,
respectively. It is well-known, see \cite[Theorem 2.2]{BF}  
and \cite{Gor1}, that there is an isomorphism of functors 
\[\Ind \cong \Coind(  \Lambda^{\text{top}}\mf g_\ob\otimes-).\] 
\vskip 0.2cm
  
\subsection{Lie superalgebras of type I} \label{subsect::21}

A classical Lie superalgebra   $\mf g = \mf g_\oo\oplus \mf g_{\ob}$  
is said to be of {\em type I} if $\mf g$ has a compatible 
$\mathbb{Z}$-grading $\mf g= \mf g_{-1}\oplus \mf g_0 \oplus \mf g_{1}$
such that $\mf g_0 =\mf g_\oo$ and  
$\mf g_\ob = \mf g_1 \oplus \mf g_{-1}$, where $\mf g_{\pm 1}$ 
are $\mf g_\oo$-submodules of $\mf g_\ob$ with 
$[\mf g_{1},\mf g_{ 1}] =[\mf g_{-1},\mf g_{-1}]  =0$. 
We set $\mf g_{\geq 0}: = \mf g_0 \oplus \mf g_1$ and 
$\mf g_{\leq 0}: = \mf g_0 \oplus \mf g_{-1}$.
 
\subsubsection{Triangular decompositions for $\mf g$ of type I} 

For classical Lie superalgebras $\mf g$ of type I considered 
in the present paper, we always assume that the triangular 
decomposition in \eqref{eq::tridec} has the property 
\begin{align}
&\mf b = \mf b_\oa \oplus \mf g_1, \label{eq::typeItridec}
\end{align}   
that is, $\mf b_\ob=\g_1$. 

In particular, we are interested in the following  
superalgebras from Kac's list (cf. \cite{Ka}): 
\begin{align} 
&(\text{Type \bf A}): ~\mf{gl}(m|n), \mf{sl}(m|n) 
\text{ and } \mf{sl}(n|n)/\mathbb{C}I_{n|n};\label{cLIa} \\
&(\text{Type \bf C}): ~\mf{osp}(2|2n); \label{cLIc} \\
&(\text{Type \bf P}): ~\pn\text{ and } [\pn, \pn]; \label{cLIp}
\end{align} 
where $m>n \geq 1$ are integers (see, e.g., 
\cite[Chapter 2, 3]{Mu12} and \cite[Section 1.1]{ChWa12}). 
See, Section~\ref{sect::32} and Section~\ref{eq::osp} for more details 
on $\pn$ and $\mf{osp}(2|2n)$.

\subsubsection{Induced and coinduced modules}

For a given object $V \in \mc O_\oa$, we extend $V$ 
trivially to a $\mf g_{\geq 0}$-module  and define 
the {\em Kac module} $$K(V):=\Ind_{\mf g_{\geq 0}}^{\mf g}V.$$ 
This defines an exact functor 
$K({}_-): \mc O_\oa \rightarrow \mc O$ which is called the 
Kac functor. For $\la \in \h^\ast$, we define  
$K(\la) := K(L_\oa(\la))$.  Also, we note that 
$\Delta(\la)  \cong K(\Delta_\oa(\la))$ and  
$\nabla(\la):= \text{Coind}_{\mf g_{\leq 0}}^{\mf g} (\nabla_\oa(\la)).$

\section{Homomorphisms between Verma supermodules} \label{Sect::HomVerSuper}

In  this section  we assume that 
\begin{enumerate}
\item $\mf g$ is a classical  Lie superalgebra of type I 
such that the triangular decomposition in \eqref{eq::tridec} 
satisfies the condition in \eqref{eq::typeItridec}. \label{eq::cond1ch3}
\item There is an   element $d\in \mc Z(\mf g_\oa)$ and a  non-zero scalar $D\in \C$ such that $$[d,x]=Dx,~ \text{for any} ~ x\in \mf g_{-1}.$$  \label{eq::cond2ch3}
\end{enumerate}  
On can check that the superalgebras $\mf{gl}(m|n)$, $\mf{osp}(2|2n)$ and $\mf{pe}(n)$ 
from the list \eqref{cLIa}-\eqref{cLIp} satisfy conditions \eqref{eq::cond1ch3} and \eqref{eq::cond2ch3}. We refer to \cite[Section 4.2]{CM} for a description of the 
elements $d$.

One of the goals in this subsection is to give a necessary and sufficient condition for the homomorphism  
$$K({}_-): \text{Hom}_{\mc O_\oa}(\Delta_\oa(\mu),\Delta_\oa(\la)) \longrightarrow \text{Hom}_{\mc O}(\Delta(\mu),\Delta(\la)),$$ 
where $K({}_-)$ is the Kac functor defined in 
Subsection~\ref{subsect::21}, to be an isomorphism.  
 
\subsection{Kac functor and homomorphisms between Verma supermodules} 
Using the same argument as in the proof of  \cite[Lemma 4.3]{CP}, 
we have the following useful lemma.

\begin{lem}\label{lem::CP43} 
Let $\mu, \nu \in \h^\ast$ be integral weights such that  $[K(\mu):L(\nu)]>0$  and $K(\nu )=L(\nu)$. Then $\mu=\nu$. 
\end{lem}

Our  first main result is the following. 

\begin{thm} \label{mainthm}
Let $\la\in \h^\ast$ be integral. Then the following 
conditions are equivalent:
\begin{itemize}
\item[(1)] The Kac module $K(\uv \la)$ is simple, where $\uv \la \in W\cdot \la$ is anti-dominant.
\item[(2)] For any $\mu \in \h^\ast$, every non-zero homomorphism 
$f:\Delta(\mu)\rightarrow \Delta(\la)$ is an embedding. 
\item[(3)] For any $\mu \in \h^\ast$, the Kac functor $K({}_-)$ 
gives rise to an isomorphism,
\begin{align*}
&K({}_-): \emph{Hom}_{\mc O_\oa}(\Delta_\oa(\mu),\Delta_\oa(\la)) \xrightarrow{\cong} \emph{Hom}_{\mc O}(\Delta(\mu),\Delta(\la)).\end{align*}  
\end{itemize} 
\end{thm}

\begin{proof} 
In the proof we regard $\Delta_\oa(\la)$ and $\Delta_\oa(\mu)$ 
as direct summands of $\Res \Delta(\la)$ and 
$\Res \Delta(\mu)$, respectively.
	
We first prove $(1)\Rightarrow (3)$. Using the exactness of 
Kac functor, the characters of $\Delta(\la)$ and $\Delta(\mu)$ 
are given as follows:
\begin{align*}
&\text{ch}\Delta(\la) = \sum_{\zeta} [\Delta_\oa(\la): L_\oa(\zeta)] \text{ch}K(\zeta),\qquad
\text{ch}\Delta(\mu) = \sum_{\zeta} [\Delta_\oa(\mu): L_\oa(\zeta)] \text{ch}K(\zeta).
\end{align*}  
This implies that there exists $\zeta \in W\cdot \mu$ such that 
$[K(\zeta):L(\uv \la)]>0$ since $f$ is non-zero and 
$\text{soc}(\Delta(\la)) =L(\uv \la)$  by \cite[Theorem~51]{CM}. 
	
Since $d\in \mc Z(\mf g_\oa)$, the element $d$ acts on 
$\Delta_\oa(\la)$ and $\Delta_\oa(\mu)$ via some scalars,
say $c_\la$ and $c_\mu$, respectively. From Lemma~\ref{lem::CP43},
we conclude that $\zeta =\uv \la$ since 
$[K(\zeta):L(\uv \la)]>0$  and $K(\uv \la )=L(\uv \la)$. 
Therefore $\mu \in W\cdot \la$ and hence $c_\la =c_\mu$.
	
Applying the restriction functor to $f$, we obtain   a 
$\g_\oa$-homomorphism  
$$\Res(f): \Lambda^{\bullet} \mf g_{-1} \otimes \Delta_\oa(\mu) \rightarrow \Lambda^{\bullet} \mf g_{-1}\otimes \Delta_\oa(\la).$$ 
We now consider the eigenspaces with respect to the operator 
$d$.  Observe that, for any $k\geq 0$, the subspace  
$\Lambda^{k} \mf g_{-1}\otimes \Delta_\oa(\mu)$
of $\Res \Delta(\mu)$ is the eigenspace for $d$ 
corresponding to the eigenvalue $c_\mu+kD$. Similarly, the subspace
$\Lambda^{k} \mf g_{-1}\otimes \Delta_\oa(\la)$  
of $\Res \Delta(\lambda)$ is the eigenspace for $d$  
corresponding to the with eigenvalue $c_\la+kD$. Therefore, we have
\begin{align} \label{eq::11}
0\neq \Res(f)(\Delta_\oa(\mu)) \subseteq \Delta_\oa(\la)
\end{align} 
since $f\neq 0$ and $c_\la =c_\mu.$ Consequently, we may conclude that 
\begin{align} \label{eq::12}
\Res(f)\bigg{|}_{\Delta_\oa(\mu)}: \Delta_\oa(\mu) \rightarrow  \Delta_\oa(\la)
\end{align} 
is an embedding by \cite[Theorem~7.6.6]{Di}. 
	
Let $K$ be the kernel of $f$. If we assume that $f$ is not an 
embedding, then we have  $K \supseteq \text{soc}(\Delta(\mu))$ by 
\cite[Theorem~51]{CCM}. This gives 
\begin{align}
&K \supseteq \text{soc}(\Delta(\mu)) = 
U(\mf g)\cdot (\Lambda^{\text{top}}\mf 
g_{-1}\otimes \text{soc}(\Delta_\oa(\mu))),
\end{align} 
which means that 
\[\Lambda^{\text{top}}\mf g_{-1}\otimes \Res(f)(\text{soc}
(\Delta_\oa(\mu)))=f(\Lambda^{\text{top}}\mf g_{-1}\otimes 
\text{soc}(\Delta_\oa(\mu))) =0,\] 
a contradiction to \eqref{eq::11}. Note that this proves the
direction  $(1)\Rightarrow (2)$. We proceed, however, with 
the proof of the direction  $(1)\Rightarrow (3)$. 
	
Observe that the Kac functor gives rise to a 
monomorphism 
$$K({}_-): \text{Hom}_{\mc O_\oa}(\Delta_\oa(\mu),\Delta_\oa(\la)) \rightarrow \text{Hom}_{\mc O}(\Delta(\mu),\Delta(\la)).$$  
It suffices to show that, for given non-zero 
$f,g\in \text{Hom}_{\mc O}(\Delta(\mu),\Delta(\la))$, 
there exists $\alpha \in \C$ such that $f=\alpha g.$ 
By Equation~\eqref{eq::12}, both $f$ and $g$ restrict 
to non-zero $\mf g_\oa$-homomorphisms  
$\Res(f)\bigg{|}_{\Delta_\oa(\mu)}$ and 
$\Res(g)\bigg{|}_{\Delta_\oa(\mu)}$ from  $\Delta_\oa(\mu)$ 
to  $\Delta_\oa(\la)$, respectively. From \cite[Theorem~7.6.6]{Di}
we have that there 
exists $\alpha \in \C$ such that  
$$\Res(f)\bigg{|}_{\Delta_\oa(\mu)} =\alpha \Res(g)\bigg{|}_{\Delta_\oa(\mu)}.$$
This proves the part $(1)\Rightarrow(3)$   since 
$\Delta(\mu)$ is generated by $\Delta_\oa(\mu)$ as a $\mf g$-module. 
	
Note that the implication $(3)\Rightarrow (2)$ is clear. 
It remains to prove the implication $(2)\Rightarrow (1)$,
so we now assume  $(2)$. 
We set $L(\nu):=\text{soc}(K(\uv \la))$. Then we have the map 
$\Delta(\nu) \onto L(\nu) \hookrightarrow K(\uv \la)= \Delta(\uv \la)$.
Since this map is an embedding by $(2)$, it follows that
$\Delta(\nu) = L(\nu)  \hookrightarrow K(\uv \la)= \Delta(\uv \la)$. 
By Lemma \ref{lem::CP43}, from $[K(\uv \la):L(\nu)]\neq 0$ and 
$K(\nu) =L(\nu)$, we deduce that $\uv \la = \nu$. 
This gives $K(\uv \la) =L(\uv \la)$ and  completes the proof. 
\end{proof}

\subsection{Example: The periplectic Lie superalgebras $\pn$} \label{sect::32} 
\subsubsection{} \label{sect::321}
For positive integers $m,n$, the general linear Lie superalgebra $\mathfrak{gl}(m|n)$ 
can be realized as the space of $(m+n) \times (m+n)$ complex matrices
\begin{align*}
\left( \begin{array}{cc} A & B\\
C & D\\
\end{array} \right),
\end{align*}
where $A,B,C$ and $D$ are $m\times m, m\times n, n\times m, n\times n$ matrices,
respectively. The Lie bracket of $\gl(m|n)$ is given by the super commutator. 
Let $E_{ab}$, for $1\leq a,b \leq m+n$ be the elementary matrix in $\mathfrak{gl}(m|n)$. 
Its $(a,b)$-entry  is equal to $1$ and all other entries are $0$.

\subsubsection{}

The standard matrix realization of the periplectic Lie superalgebra 
$\pn$ inside the general linear Lie superalgebra 
$\mathfrak{gl}(n|n)$ is given by
\begin{align}\label{plrealization}
\mf g= \mf{pe}(n):=
\left\{ \left( \begin{array}{cc} A & B\\
C & -A^t\\
\end{array} \right)\| ~ A,B,C\in \C^{n\times n},~\text{$B^t=B$ and $C^t=-C$} \right\}.
\end{align}
Throughout the present paper, we fix the Cartan subalgebra $\mf h= \mf h_\oa \subset \mf g_\oo$ consisting of diagonal matrices. We denote the dual basis of $\mf h^*$ by $\{\vare_1, \vare_2, \ldots, \vare_n\}$ with respect to the standard basis of $\mf h$ defined as
\begin{align}
\{H_i:=E_{i,i}-E_{n+i,n+i}|~1\leq i \leq n \}\subset \pn, \label{eq::cartan}
\end{align}  
where $E_{a,b}$ denotes the $(a,b)$-matrix unit, 
for $1\leq a,b \leq 2n$.
The sets $\Phi_\oa$, $\Phi_\ob^\pm$ and $\Phi_\ob$ of even,
odd positive, odd negative and odd roots are, respectively,  given by
\begin{align}\label{eqroots}
&\Phi_\oa=\{\epsilon_i-\epsilon_j\,|\, 1\le i\not=j\le n\}, \\
&\Phi^+_\ob=\{\epsilon_i+\epsilon_j\,|\, 1\le i\le j\le n\} ,\\
&\Phi^-_\ob=  \{- \epsilon_i-\epsilon_j\,|\, 1\le i< j\le n\},\\  
&\Phi_\ob:=\Phi_\ob^+\cup \Phi_\ob^-.
\end{align} 

The space $\mf h^\ast = \bigoplus_{i=1}^n\C \vare_i$ is equipped with 
a natural bilinear form  
$\langle \vare_i , \vare_j \rangle =\delta_{i,j}$, for any 
$1\leq i,j \leq n$. The Weyl group $W=\mf S_n$   is the symmetric 
group    acting on $\h^\ast$.  For any $\alpha\in\Phi_\oa$, we let 
$s_\alpha$ denote the corresponding reflection in $W$.  

\subsubsection{} 

We  fix the Borel subalgebra $\fb_{\oa}$ of $\g_{\oa}=\mathfrak{gl}(n)$
consisting of matrices in \eqref{plrealization} with $B=C=0$ and
$A$ upper triangular. 

The subalgebras $\mf g_1$ and $\mf g_{-1}$ are given by 
\begin{align*}
\mf g_1:=
\{\begin{pmatrix}
0 & B \\
0 & 0
\end{pmatrix}|B^t=B\}\quad\mbox{and}\quad \mf g_{-1}:=
\{\begin{pmatrix}
0 & 0 \\
C & 0
\end{pmatrix}|C^t=-C\}.
\end{align*}
We note that $\Phi_\ob^+$ and $\Phi_\ob^-$ are the sets of 
weights of the $\mf g_0$-modules $\mf g_1$ and $\mf g_{-1}$, respectively. 

The Borel subalgebra is defined as $\mf b:=\mf b_\oa \oplus \g_{1}$, while the reverse Borel subalgebra $\mf b^r$ is defined as 
$\mf b^r:=\mf b_\oa \oplus \mf g_{-1}.$  It can be shown 
that $\mf b^r$ is still a Borel subalgebra in thee sense 
of \cite[Section~2.4]{Ma} (see, e.g., \cite[Section~5]{CCC}).
 
Furthermore, we let $L^{r}(\la)$ be the irreducible module 
in the corresponding ``reversed'' category $\mc O$ of $\mf b^r$-highest weight $\la$. 
We let  $\la^+\in \h^\ast$ be uniquely determined by
\begin{align}
&L(\la^+) =L^r(\la). \label{eq::oddref}
\end{align}
It is known that the  weight $\la^+$  can be computed using 
odd reflection, and we refer to \cite[Section~2.2]{PS89} and \cite[Theorem 3.6.10]{Mu12}
for a treatment of odd reflections for the periplectic Lie superalgebras $\pn$. 
In particular,  the effect on the highest weight of a 
simple $\pn$-module under odd reflection and inclusion 
was computed in \cite[Lemma~1]{PS89}.

Finally, we set  $\displaystyle X:=\bigoplus_{i=1}^n\mZ\ep_i$;  
$\omega_k:=\vare_1+\cdots +\vare_k$, for any $1\leq k\leq n$; 
and also $\eta:=(1-n)(\vare_{1}+\cdots +\vare_{n})$. 
Then we have $\Lambda^{\text{top}}\mf g_{-1}\cong \C_{\eta}$ 
as a $\mf g_\oa$-module. 

\subsubsection{}

Let $\la, \mu\in \h^\ast$ be two integer weights.
If there are  positive roots 
$$\alpha_1,\alpha_2,\ldots, \alpha_k\in \Phi_\oa^+,$$ 
such that 
$\la =s_{\alpha_k} s_{\alpha_{k-1}}\cdots s_{\alpha_1}\cdot \mu $ 
and, for each $1\leq q \leq k-1$, we have  
\[\langle s_{\alpha_q}s_{\alpha_{q-1}}\cdots s_{\alpha_1}\cdot \mu,~ \alpha_{q+1} \rangle \geq 0, \] 
then  we write $\mu \uparrow \la$, cf. \cite[Subsection~7.6]{Di}.

\begin{prop} \label{matincoro0}
Let $\mf g=\pn$ and $\la \in \h^\ast$ be an integral weight.  
\begin{itemize} 
\item[(1)] For any integral weight $\mu\in \h^\ast$, we have that  
$$\emph{Hom}_{\mc O}(\Delta(\mu),\Delta(\la)) \cong \begin{cases}
\C, \emph{ if } \mu \uparrow \la,\\
0, \emph{ otherwise. }
\end{cases}$$
\item[(2)] A weight vector $v\in \Delta(\la)$ satisfies $\mf nv=0$ 
if and only if $$v=1\otimes v_0\in  U(\mf g) \otimes_{\mf g_{\geq 0}} 
\Delta_\oa(\la) \equiv  \Res \Delta(\la),$$ 
for some vector $v_0$ satisfying $\mf n_\oa v_0=0$. 
\end{itemize} 
\end{prop} 

\begin{proof}
We recall that $\pn$ admits a grading operator   
$d:=\text{diag$(1,1,\ldots,1)$}\in \g_\oa$ which satisfies 
Condition~\eqref{eq::cond1ch3} at the beginning of 
Section~\ref{Sect::HomVerSuper}. Also, it follows from 
\cite[Lemma~5.11]{CC} that every weight satisfies 
Assumption $(1)$ in Theorem~\ref{mainthm}. Therefore the proof follows 
from Theorem~\ref{mainthm} and \cite[Theorem~4.2, Theorem~5.1]{Hu08}. 
\end{proof}

\begin{rem}
	In \cite{LLW}, Liu, Li and Wang  established a closed formula for a singular vector of weight $\la -\beta$ in the Verma
	module of highest weight $\la$ for Lie superalgebra $\gl(m|n)$ when $\la$ is atypical with respect to an odd positive root $\beta$. In particular, the authors proved the following identity
	\begin{align}
	&\text{Hom}(\Delta(\la -\beta),\Delta(\la))=\C, \label{eq::LLWid}
	\end{align}  where $\Delta$'s denote Verma supermodules for $\g =\gl(m|n)$. This identity was also known in \cite{ChWa18} for the exceptional Lie superalgebra $D(2|1;\zeta)$. By Proposition \ref{matincoro0}, it follows that the identity \eqref{eq::LLWid} fails to hold for $\pn$. 
\end{rem}

\begin{rem}  The Borel subalgebras have been classified  in \cite[Section 5]{CCC}. It would be interesting to determine for which other Borel  subalgebras one has an analogue of
	Proposition \ref{matincoro0}. We note that Proposition \ref{matincoro0} fails in 	the case of the reverse Borel subalgebra $\mf b^r$. To see this, we define Verma supermodules with respect to $\mf b^r$ as follows
	$$\Delta^r(\la):=U(\mf g)\otimes_{\mf b^r}\C_\la \cong \text{Ind}^\g_{\mf g_{\leq 0}}\Delta_\oa(\la),$$
	which admits a filtration with subquotients   $ \text{Ind}^\g_{\mf g_{\leq 0}}L_\oa(\zeta)$, for $\zeta \in \h^\ast$.	By \cite[Proposition 4.15]{CM}, the modules $\text{Ind}^\g_{\mf g_{\leq 0}}L_\oa(\zeta)$ are never simple, for any $\zeta\in \h^\ast$. By similar argument as used in the direction $(3)\Rightarrow (2)$ and $(2)\Rightarrow (1)$ in the proof of Proposition \ref{mathcoro}, it follows that  Proposition \ref{mathcoro} fails for these modules $\Delta^r(\la)$.
\end{rem}

\begin{rem}
Consider a basic Lie superalgebras $\mf g$ of type I, for instance,   
$\mf g=\gl(m|n)$, or $\mf{osp}(2|2n)$. It has been shown in  by Kac \cite{Ka2}
that a weight $\la$ satisfies Assumption (1) in Theorem~\ref{mainthm}  
if and only if $\la$ is typical. Since the notions of strongly 
typical and typical are identical in these cases
 (see, e.g., \cite[Section 2.4]{Gor2} or \cite{Gor3}), the statement 
of Theorem~\ref{mainthm} follows from Gorelik's  theorem, 
see \cite[Theorem 1.3.1, Theorem 1.4.1]{Gor2}.  
\end{rem}

\section{The socle of the cokernel of inclusions of Verma supermodules 
for  $\pn$} \label{Sect::soccokerpn}
 
In this section, we assume that $\mf g=\pn$.  The goal of this 
section is to describe the socle of the cokernel of an arbitrary 
non-zero homomorphism between Verma $\pn$-supermodules. 

Let $\la,\gamma\in \h^\ast$ and
$\Delta(\gamma) \xrightarrow{f} \Delta(\la)$ be a
non-zero homomorphism between the corresponding
Verma supermodules  over $\mf g$.  
By Theorem \ref{mainthm}, there exists a homomorphism 
$\tilde{f}\in \text{Hom}_{\mc O}(\Delta_\oa(\gamma), \Delta_\oa(\la))$ 
such that $K(\tilde{f}) =f$. Applying the Kac functor to the  
short exact sequence   
$$0\rightarrow \Delta_\oa(\gamma) \xrightarrow{\tilde{f}} \Delta_\oa(\la) \rightarrow \Delta_\oa(\la)/\Delta_\oa(\gamma)\rightarrow 0$$
in $\mc O_\oa$,
we obtain the  short exact sequence  
$$0\rightarrow \Delta(\gamma) \xrightarrow{f} \Delta(\la) \rightarrow K(\Delta_\oa(\la)/\Delta_\oa(\gamma))\rightarrow 0$$ 
in $\mc O$. This  gives rise to the isomorphism  
$$\text{soc}(\Delta(\la)/\Delta(\gamma)) \cong 
\text{soc}(K(\Delta_\oa(\la)/\Delta_\oa(\gamma))) ,$$
where we write $\Delta(\la)/\Delta(\gamma)$ for
$\Delta(\la)/f(\Delta(\gamma))$.
	
Also, by Theorem~\ref{mainthm}, there exists a dominant integral 
$\mu$ such that  $\la = x\cdot \mu$ and $\gamma =y\cdot \mu$, 
for some $x,y\in W$ with $x<y$ with respect to the Bruhat order. 
Let 
$$\theta_\mu^{\text{on}}:(\mc O_\oa)_0 \rightarrow (\mc O_\oa)_\mu$$ 
denote the translation functor to the $\mu$-wall, where 
$(\mc O_\oa)_0$ and $(\mc O_\oa)_\mu$ denote the principal 
block and the block containing $L_\oa(\mu)$, respectively.  
For any $ s\in W$, we set  $n_{x, y, z}$ to be the integer 
determined by    
$$\theta_\mu^{\text{on}}\left(\text{soc}(\Delta_\oa(x'\cdot 0)/\Delta_\oa(y'\cdot 0))\right)=\text{soc}(\Delta_\oa(x\cdot \mu)/\Delta_\oa(y\cdot \mu))=\bigoplus_{z\in W}L_\oa(z\cdot \mu)^{\oplus n_{x, y, z}},$$
where $x',y'$ are the shortest representatives in the cosets $xW_\mu$ and 
$yW_\mu$, respectively, see \cite[Proposition~15]{KKM}.  
 
For $\nu \in \h^\ast$ we recall the weight $\nu^+$ defined in
\eqref{eq::oddref}. We are now in a position to state the main 
result of this section which describes the socle of the cokernel 
of homomorphism between Verma supermodules  over $\mf g$.

\begin{thm} \label{2ndmainthm}   
We have the following description of 
the socle of the cokernel of a non-zero homomorphism 
$\Delta(\gamma) \rightarrow \Delta(\la)$:
\[\emph{soc}(\Delta(\la)/\Delta(\gamma))\cong   
\bigoplus_{z\in W}L((z\cdot \mu)^++\eta)^{\oplus n_{x,y,z}}. \]
\end{thm}

\begin{proof}    
By \cite[Lemma~3.2]{CM} (see also \cite[Theorem 51]{CCM}), we have  \begin{displaymath}
\begin{array}{rcl}
\text{soc}(\Delta(\la )/\Delta(\gamma)) 
&\cong &\text{soc}(K(\Delta_\oa(\la)/\Delta_\oa(\gamma))) \\
&= &U(\mf g)\cdot  \left(\Lambda^{\text{top}} \mf g_{-1}\otimes 
\text{soc}(\Delta_\oa(\la)/\Delta_\oa(\gamma))\right) \\ 
&\cong& \text{head}\,\Ind_{\g_{\leq 0}}^\g 
\left(\bigoplus\limits_{z\in W}
L_\oa(z\cdot 0 +\eta)^{\oplus n_{x,y,z}}\right).
\end{array}
\end{displaymath}
It follows that 
$$\text{soc}(\Delta(\la)/\Delta(\gamma))\cong 
\bigoplus\limits_{z\in W}L^r(z\cdot 0 +\eta)^{\oplus n_{x,y,z}}= 
\bigoplus_{z\in W}L((z\cdot \mu)^++\eta)^{\oplus n_{x,y,z}}.$$  
This completes the proof. 
\end{proof}

\begin{rem}
Let $\la$ be a dominant and regular weight. 
By Theorem~\ref{2ndmainthm}, for any $y\in W$, the quotient 
$\Delta(\la)/\Delta(y\cdot \la)$ has simple socle if and only if 
$y$ is a bigrassmannian permutation, see 
\cite[Section~1.1, Theorem~1.3]{KKM}.
\end{rem}

For an integral weight $\la \in \h^\ast$, we let 
$\ov \la\in W\cdot \la$ and 
$\uv \la\in W\cdot \la$ be the unique dominant and 
the unique anti-dominant element, respectively.   
Theorem~\ref{2ndmainthm} has the following consequence.

\begin{cor} \label{mathcoro} 
Suppose that $\mu \in \h^\ast$ is not anti-dominant. Then we have  
$$\emph{dim Ext}_{\mc O}^1(L(\mu),\Delta(\la)) = 
[\emph{soc}\left(\Delta(\ov \la)/\Delta(\la)\right): L(\mu)].$$
\end{cor}

\begin{proof}
Let  
\begin{align} \label{eq::SES1}
&0\rightarrow \Delta(\la) \rightarrow M \rightarrow L(\mu) \rightarrow 0,
\end{align} 
be a non-split short exact sequence. Since it is non-split, 
the socle of $M$ 
coincides with the socle   $L(\uv \la)$ of $ \Delta(\la).$ 
Therefore $M$ is a submodule of the injective envelope 
$I(\uv \la) =T(\ov \la)$. Now, we consider a two step 
projective-injective copresentation of $\Delta(\ov \la)$
(whose existence follows, for example, from \cite[Theorem~7.3]{AM}): 
\[0\rightarrow \Delta(\ov \la) \rightarrow T(\ov \la) \xrightarrow{f} U. \]
By our assumption, $[M/\text{soc}(M): L(\zeta)] >0$ implies that $\zeta$ 
is not antidominant. As a consequence, $f(M)=0$ since the socle 
of the projective-injective module $U$ is a direct sum of simple 
modules of anti-dominant   highest weights. This means that $M$ 
is a submodule of $\Delta(\ov \la)$. Consequently, 
$L(\mu)\cong M/\Delta(\la)$ corresponds to a socle constituent of 
$\Delta(\ov \la)/\Delta(\la)$. This completes the proof.  
\end{proof}

\begin{ex}
We consider $\mf g= \mf{pe}(2).$ Let $\la, \mu\in \h^\ast$ be integral. Suppose that  $\mu$ is not antidominant.  Set $\ov \la =a \vare_2+b\vare_2$. Then, by \cite[Lemma 6.1]{CC}, we have   
\begin{align} 
\text{soc}(\Delta(\ov \la) /\Delta(\la))=
\begin{cases}
L(\ov \la-\omega_2), \text{ if $a =b$.} \\
L(\ov \la), \text{ if $a >b$.}\\
0,\text{ if $\la =\ov \la.$}
\end{cases}
\end{align} 
It follows that  
\begin{align*}
&\text{dim}\,\text{Ext}^1_{\mc O}(L(\mu), \Delta(\la)) = \begin{cases}
1, \text{ if either $\mu =\ov \la-\omega_2$, for $a=b$, or $\mu=\ov \la$, for $a>b$};\\
0, \text{ otherwise.}
\end{cases}
\end{align*}
\end{ex}

\begin{rem}   
Proposition~\ref{2ndmainthm} and Corollary~\ref{mathcoro} can be generalized to other classical Lie superalgebras of type I which 
have type A even part, under the condition that 
$\text{dim}\,\text{Hom}_{\mc O}(\Delta(\gamma), \Delta(\la))\leq 1$.   
\end{rem}

\section{Homological dimensions} \label{Sect::homodim}

Let $\mc A$ be an abelian category. Suppose that $\mc A$ has 
enough projective and injective objects. For any $M\in \mc A$, 
we let $\pd_{\mc A}M$   denote the projective dimension of  $M$. 
We define $\gld\,\mc A$ to be the global dimension of  $\mc A$, namely, 
\[\gld\,\mc A:=\mathop{\text{sup}}\limits_{M \in\mc A} \{\pd_{\mc A} M\}.\]
The finitistic dimension of $\mc A$ is defined as 
\[ \text{fin.dim}\,\mc A : =\mathop{\text{sup}}\limits_{M \in \mc A} \{\pd_{\mc A}M|~\pd_{\mc A}M<\infty \}. \] 
We will simply use $\pd({}_-)$ and $\gld({}_-)$ when the context is clear. Similarly, we have the injective dimension $\id\,M= \id_{\mc A}M$ 
of an object $M\in\mc A$. 
 
In this section, we assume that $\mf g$ is an arbitrary classical Lie superalgebra. 

\subsection{Finitistic dimension of various categories of supermodules}  \label{sect::51}  

We denote by $\mf g$-Mod and $\mf g_\oa$-Mod the category 
of all $\mf g$-modules and the category of all 
$\mf g_\oa$-mo\-du\-les, respectively.  
   
Let $\tilde{\mc C}$ and  ${\mc  C}$ be abelian subcategories of 
$\mf g$-Mod and $\mf g_\oa$-Mod, respectively. Assume that 
\begin{enumerate}[(i)] 
\item $\Ind, \Coind$ and $\Res$ restricts to well-defined functors between ${\mc C}
$ and $\tilde{\mc C}$.  \label{eq::111}
\item $\mc C$ has enough injective and projective modules.  \label{eq::112}
\item $\text{gl.dim}\,\mc C <\infty$.\label{eq::113}
\end{enumerate}

 The following theorem generalizes \cite[Theorem~4.1]{CS} (also see \cite[Lemma 4.3]{CS}) where the case
of the superalgebra $\gl(m|n)$ was considered.
\begin{thm} \label{thm::findimthm}
We have $\emph{fin.dim} \,\tilde{\mc C} = \emph{gl.dim}\,\mc C$.
\end{thm} 
\begin{proof} 
Since $\Ind$ is isomorphic to $\Coind$, up to parity
change, it follows that the functors $\Ind$, 
$\Coind$ and $\Res$ send injective
(resp. projective) objects to injective (resp. projective) object.

The category $\tilde{\mc C}$ is a Frobenius extension of $\mc C$ in 
the sense of \cite[Definition~2.2.1]{Co}. Therefore   $\tilde{\mc C}$ 
has enough projective and injective objects by 
\cite[Proposition~2.2.1]{Co}. Moreover, every injective 
(resp. projective) object in 
$\tilde{\mc C}$ is a direct summand of an object induced 
from an injective (resp. projective) object of $\mc C$. 

Consequently, all injective objects in $\tilde{\mc C}$
have finite projective dimension.
By an argument used in the proof   \cite[Theorem~3]{Ma} 
(see also  \cite[Corollary~3.6.10]{Co}), we have that
$\text{fin.dim}\, \tilde{\mc C}$ is the maximum of projective 
dimensions of injective objects in $\tilde{\mc C}$. Note also that the
global dimension of $\mc C$ is the maximum of projective 
dimensions of injective objects in $\mc C$, as well.

We observe that, for any injective 
$I$ in $\mc C$ and $Y\in \tilde{C}$, it follows from the 
Frobenius reciprocity  that 
\[\Ext^d_{\tilde{\mc C}}(\Ind I, Y) =\Ext^d_{\mc C}(I,\Res Y),  \]
which implies that 
$\text{fin.dim} \,\tilde{\mc C}  \leq  \text{gl.dim}\,\mc C$. 
   	
Finally, let $I\in \mc C$ be injective. Then $I$ is a direct
summand of $\Res \Ind I\cong \Lambda^{\bullet} \mf g_\ob \otimes I$, 
which implies that 
\[\pd\, I \geq \pd\, \Ind I \geq \pd \,\Res \Ind I \geq \pd\, I. \] 
This shows that 
$\text{fin.dim} \,\tilde{\mc C}  \geq  \text{gl.dim}\,\mc C$ 
and the claim follows. 
\end{proof}
   
Below we apply Theorem~\ref{thm::findimthm} to various examples.

\begin{ex}
Let  $\tilde{\mc  C}= \mf g$-Mod and ${\mc  C}= g_\oa$-Mod. 
Condition~\eqref{eq::111} is obvious and Conditions~\eqref{eq::112}
and \eqref{eq::113} are clear since $U(g_\oa)$ is noetherian.
We have $\text{gl.dim}\,\mc C =\text{dim}\,\mf g_\oa$.   Therefore, 
by Theorem~\ref{thm::findimthm}, we have 
$$\text{fin.dim}~\mf g\text{-Mod} = \dim \g_\oa.$$ 
\end{ex}
   
\begin{ex}
Consider the super analog $\mc C(\mf g,\h)$ of the category 
$\mc C(\mf g_\oa,\h)$ of weight $\mf g$-modules from \cite{Zu}, 
that is, $\mc C(\mf g,\h)$ the full subcategory of $\mf g$-Mod 
consisting of $\mf g$-modules which are semisimple over $\mf h$.

For $\la \in \h^\ast$, let $I_\la$ be the left ideal of 
$U(\mf g)$ generated by $h-\la(h)$, for $h\in\mf h$.
Then the quotient of $U(\mf g)$ by $I_\la$ 
is projective in $\mc C(\g,\h)$. Also, every weight $\mf g$-module 
is a quotient of a direct sum of the $I_\la$'s. Therefore 
$\mc C(\mf g_\oa, \h)$ has enough projective objects. By 
\cite[Lemma 2.2]{Zu}  $\mc C(\mf g_\oa,\mf h)$ has 
enough injective objects as well. 
   	
Let $M\in \mc C(\mf g_\oa,\h)$. By \cite[Corollary 3.1.8]{Ku}, 
the trivial module $\C$ in $\mc C(\mf g_\oa,\h)$ has the following
finite projective resolution  in $\mc C(\mf g_\oa,\h)$: 
\[\cdots  \rightarrow D_2\rightarrow D_1 \rightarrow D_0 \rightarrow \C \rightarrow 0, \]
where $D_p$ is the induced module  $U(\mf g_\oa)\otimes_{U(\mf h)}\Lambda^p(\mf g_\oa/\mf h)$, for $p\geq  0$. Tensoring with this 
resolution (over $\mathbb{C}$), we conclude that every  
$M\in \mc C(\mf g_\oa,\mf h)$ 
has finite projective dimension. 
   	
In particular, we have 
\begin{align*}
&\text{fin.dim}\, \mc C(\mf \g,\h)=
\text{fin.dim} \,\mc C(\mf \g_\oa,\h) =
\text{dim}\g_\oa -\text{dim}\h.
\end{align*}
\end{ex}
   
\begin{ex}
Let $\mc O^\fp$ be the parabolic BGG category of $\mf g$-modules 
in the sense of \cite[Section 3]{Ma}. By \cite{CoM2}, we have 
\begin{align} \label{eq::egfindim}
&\text{fin.dim}\, \mc O^\fp  =
\gld \,\mc O_\oa^\fp=2\ell(w_0) -2\ell(w_0^\fp),
\end{align} 
where $w_0$ is the longest element in $W$ and
$w_0^\fp$ is the longest element in the Weyl group of 
the Levi subalgebra of $\fp$.
\end{ex}
   
\begin{ex}
Consider $\mf g=\pn$ with a reduced parabolic subalgebra $\fp$. 
For $\la \in \h^\ast$, let $\text{Irr}\,\mc O_{\la}^\fp$ denote 
the set of all highest weights of irreducible modules in the 
parabolic block $\mc O_\la^\fp$.   Recall the equivalence 
relation $\sim$ on $\h^*$ defined in~\cite[Subsection~5.2]{CC} 
which is transitively generated by 
$$\begin{cases}\la\sim \la \pm2\vare_k, &\mbox{ for~$1\le k\le n$;}\\
\la\sim w \cdot \la,&\mbox{ for~$w\in W_{[\lambda]}$.}
\end{cases}$$
Here   $W_{[\la]}$ denotes the integral Weyl group associated 
to $\la$. It is generated by all  reflection $s_\alpha$ with 
$\langle\la, \alpha \rangle \in \Z.$ It is well-known that 
$W_{[\la]}$ is the Weyl group of a certain semisimple Lie 
algebra (see, e.g., \cite[Theorem 3.4]{Hu08}). We define  
$w_0^{\la}$ to be the longest element of $W_{[\la]}$.    
For $\la ,\nu  \in \Sigma^+_\fp$, it follows from 
\cite[Theorem~B]{CP} (see also \cite[Theorem~3.2]{CP}) 
that  $\mc O_\la^\fp =\mc O_\nu^\fp$ if and only if $\la \sim \nu$.
As a consequence, there exists $\nu \in \text{Irr}\,\mc O_{\la}^\fp$
which is dominant and regular.
   	
Let $\la\in \h^\ast$ be a dominant weight such that the 
parabolic block $\mc O^\fp_\la$ (see Section~\ref{Sect::projdimpn} 
for the precise definition) is non-zero. Let 
$  (\mc O^\fp_{\oa})_{\la +X}$ denote the full subcategory 
of $\mc O^\fp_{\oa}$ consisting of all modules modules whose weights  
belong to the set $\la +X$. Applying, if necessary, the equivalence established in \cite[Section~2.3]{CMW} (see also 
\cite[Section 4.3]{CC}), we can assume that 
$W_{[\la]}$ is a parabolic subgroup of $W$. Then the problem
of global dimension of $(\mc O^\fp_{\oa})_{\la +X}$ reduces 
to that of an integral block for a semisimple Lie algebra of 
type $\bf A$ by \cite[Proposition~2.3]{CMW}. By  
\cite[Theorem~5.2]{CoM2} we have 
\begin{align}
&\text{gl.dim}   (\mc O^\fp_{\oa})_{\la +X} =2\ell(w_0^{\la}) 
-2\ell(w_0^\fp). \label{lem::34}
\end{align}

Take $\nu \in \text{Irr}\,\mc O_{\la}^\fp$ dominant 
and regular. In this case 
$\pd\, I_\oa^\fp(\nu) =2 \ell(w_0^{\la}) -2\ell(w_0^\fp)$ by 
\cite[Theorem~4.1, Proposition~6.9]{CoM2}.   By 
Theorem \ref{thm::findimthm}, we thus have 
\begin{align*}
&\text{fin.dim}\, \mc O_\la^\fp =2\ell(w_0^{\la}) - 2\ell(w_0^{\fp}).
\end{align*} 
\end{ex}

\subsection{The parabolic BGG category $\mc O^\fp$}
\label{Sbusect::parabolicCate}
 
In this section, we let $\mf g$ be an arbitrary  classical 
Lie superalgebra. We let $\langle {}_-,{}_-\rangle$ denote the usual 
bilinear form on $\mf h^\ast$.
 
In the remaining parts of the present paper, we use the notion of 
parabolic decomposition as defined in \cite[Section~2.4]{Ma} and 
consider a {\em reduced} parabolic subalgebra $\fp$ from  
\cite[Subsection~1.4]{CCC}. Namely, we assume that 
$\mf b\subseteq \fp\subseteq \mf g$ with a purely even 
Levi subalgebra $\mf l\subseteq \fp$.   The even part 
$\mf p_\oa$ of $\mf p$ is a parabolic subalgebra of $\mf g_\oa$ 
arising from a parabolic decomposition of $\mf g_\oa$.  
{\em In what follows, we always assume that $\fp_\ob =\g_1$ 
in the case when $\mf g$ is of type I.}
 
The {\em parabolic category} $\mc O^\fp$ (resp. $\mc O^{\fp_\oa}_\oa$) 
is the full subcategory of $\mc O$ (resp. $\mc O_\oa$) consisting
of all finitely generated $\mf g$-modules (resp. $\mf g_\oa$-modules) 
on which $\fp$ (resp. $\fp_\oa$) acts locally finitely. 
For simplicity of notation, we define 
$\mc O^\fp_\oa:= \mc O^{\fp_\oa}_\oa$, and, in what follows, 
we omit $\oa$ from   $(-)^{\fp_\oa}$ or $(-)_{\fp_\oa}$ when 
defining notation for $\g_\oa$-modules.  
 
For any $\la \in \h^\ast$, let $L^{\mf l}(\la)$ denote the 
irreducible $\mf l$-module with highest weight $\la$.
We define the set of {\em $\mf l$-dominant weights}  
(called $\fp$-dominant weights in \cite[Section~3]{CCC}) as follows:
\begin{align}
&\Sigma^+_\fp:=\{\la\in \mf h^\ast|~ \langle\la, \alpha\rangle 
\in \mathbb Z_{\geq 0},\text{ for all } \alpha \in \Phi^+(\mf l) \}. \label{eq::ldomi}
\end{align}  
The corresponding  parabolic category $\mc O^\fp$ is the Serre 
subcategory of $\mc O$ generated by 
$\{L(\la) \, | \,  \la \in \Sigma_\fp^+\}$. 
 
Let $\mc O^\fp_\la$ denote the (indecomposable) block in 
$\mc O^\fp$ containing $L(\la)$. 
In particular, $\mc O_\la:=\mc O_\la^{\mf  b}$ is the 
(indecomposable) block  containing $L(\la)$ 
in  category $\mc O$.
 
For $\la \in \Sigma_\fp^+$, we define the parabolic Verma 
$\mf g_\oa$-module and  the parabolic Verma $\mf g$-module 
respectively as  
\begin{align}
&\Delta^\fp_\oa(\la) : =
\Ind_{\mf p_\oa}^{\mf g_\oa}L^{\mf l}(\la)\quad\text{ and }\quad
\Delta^\fp(\la):=
\Ind_{\mf p}^{\mf g}L^{\mf l}(\la) . \label{Eq::paraV}
\end{align}  
Similarly, we can define the  dual parabolic Verma modules  
$\nabla^\fp_\oa(\la)$ and $\nabla^\fp(\la)$  in  
$\mc O^\fp_\oo$ and $\mc O^\fp$, respectively 
(see \cite[Definition~3.2]{CCC}). We have 
$$\Delta^\fp(\la) \twoheadrightarrow L(\la)\quad\text{ and }\quad L(\la) 
\hookrightarrow \nabla^\fp(\la).$$
 
Let $P^\fp(\la)$ be the projective cover of $L(\la)$ 
in $\mc O^\fp$ and $I^\fp(\la)$ be the injective envelope of $L(\la)$ 
in $\mc O^\fp$. For $\mc O^\fp =\mc O$, we
denote the corresponding objects by $P(\la)$ and $I(\la)$.   
Similarly, we define the indecomposable 
projectives $P_\oa(\la)$, $P_\oa^\fp(\la)$ and 
injectives $I_\oa(\la)$, $I^\fp_\oa(\la)$ in 
$\mc O_\oa$ and $\mc O_\oa^\fp$, respectively.  Also, we denote by $T^\fp(\la)$ and $T_\oa^\fp(\la)$ the tilting modules of highest weight $\la$ in $\mc O^\fp$ and $\mc O_\oa^\fp$, respectively. Finally, 
we note that $\Ind: \mc O_\oa^\fp\rightarrow \mc O^\fp$ 
and $\Res : \mc O^\fp\rightarrow \mc O_\oa^\fp$ are well-defined. 
 
Also, we note that  
$$\Delta^\fp_\la= K(\Delta^\fp_\oa(\la)),~\nabla^\fp(\la):=
\Ind_{\mf g_\oa+\mf g _{-1}}^{\mf g}(\nabla^\fp_\oa(\la)
\otimes \Lambda^{\text{top}}\mf g_1^\ast)\cong 
\Coind_{\mf p_\oa+\mf g_{-1}}^{\mf g}L^{\mf l}(\la),$$ 
in the case when $\mf g$ is of type I. 
 
We define $w_0^\fp$ to be the longest element 
in the Weyl group of $\mf l$.  Finally, we set $w_0:=w_{0}^{\mf b}$.

\subsection{Projective dimension of modules in $\mc O^\fp$} 
\label{sect::52} 

\subsubsection{Preliminary results}

The following lemma is the parabolic version  of \cite[Lemma~4.3]{CS}.

\begin{lem}\label{lem::31}
{\hspace{2mm}}

\begin{itemize}
\item[(1).] For $M\in \mc O^\fp$, we have $\epd\, M \geq \epd\, \Res M$.  
\item[(2).] For $N\in \mc O^\fp_\oa$, we have 
$\epd\, N = \epd\,  \Ind N =\epd\, \Coind N$. 
\end{itemize}
\end{lem}

\begin{proof}
As already mentioned, $\Res$, $\Ind$ and $\Coind$ send 
projective resolutions to projective resolutions. 
Also, $N$ is a direct summand of $\Res\Ind N$. 
This implies all statements.
\end{proof}

\begin{cor} \label{coro::53}
For any weight $\zeta\in \Sigma^+_\fp,$ we have 	
\[\epd\, \Delta^\fp(\zeta) \geq  \epd \, \Delta_\oa^\fp(\zeta)
\quad\text{ and }\quad
\epd\, \nabla^\fp(\zeta) \geq  \epd \, \nabla_\oa^\fp(\zeta).\] 
\end{cor}

\begin{proof}
Noting that the modules
$\Delta_\oa^\fp(\zeta)$ and  $\nabla_\oa^\fp(\zeta)$  
are  direct summands of $\Res \Delta^\fp(\zeta)$ and 
$\Res \nabla^\fp(\zeta)$, respectively, the statement follows 
from Lemma~\ref{lem::31}.
\end{proof}

The following result is a consequence of the combination of  \cite[Corollary~3.2.5]{Co} and \cite[Theorem~8.2.1]{Co}. 
We provide more details in the proof.
 
\begin{prop}\label{prop::53} 
Suppose that $\mf g$ is of type I. 
Then, for each $\la \in \Sigma^+_\fp$, we have 
$$\epd_{\mc O^\fp}\, I^\fp(\la) = 
\epd_{\mc O^\fp_\oa}\, I^\fp_\oa(\la).$$ 	
\end{prop}

\begin{proof}  
The proof of \cite[Corollary 4.7]{CCC} shows that 
$L_\oa(\la)$ is a quotient of $\Res L(\la)$. 
Therefore we have 
\begin{align*}
&\text{Hom}_{\mc O^\fp}(L(\la), \Coind\, I^\fp_\oa (\la)) = 
\text{Hom}_{\mc O_\oa^\fp}(\Res L(\la), I^\fp_\oa (\la))  = 
[\Res L(\la): L_\oa(\la)]\neq 0.
\end{align*}
This means that $L(\la) \hookrightarrow \Coind I_\oa^\fp(\la)$ and 
so $I^\fp(\la)$ must be a direct summand of $\Coind I_\oa^\fp(\la)$. 
Consequently $\pd\, I^\fp(\la) \leq \pd\, I_\oa^\fp(\la).$
	
Next, we claim that $L_\oa(\la)$ is a direct summand of the socle 
of $\Res I^\fp(\la)$. We start by observing that 
\begin{align}
&\Ind L_\oa(\la) =\Ind^\g_{\mf g_{\geq 0}} \Ind^{\g_{\geq 0}}_{\mf g_\oa} L_\oa(\la) \cong 
\Ind^\g_{\mf g_{\geq 0}} (\Lambda^{\bullet} \mf g_{1}\otimes L_\oa(\la)).
\end{align} 
This implies that 	
\begin{align*}
&\text{Hom}_{\mc O^\fp_\oa}(L_\oa(\la), \Res I^\fp (\la)) = 
\text{Hom}_{\mc O^\fp}(\Ind L_\oa(\la), I^\fp (\la))\\ 
& = [\Ind L_\oa(\la): L(\la)]\geq[\Ind^{\mf g}_{\mf g_{\geq 0}} L_\oa(\la): L(\la)] \neq 0.
\end{align*} 
Taking Lemma \ref{lem::31} into account, we thus get 
$\pd\, I_\oa^\fp(\la)\leq \pd\, \Res I^\fp(\la) \leq \pd \,I^\fp(\la)$.
This completes the proof.
\end{proof}

The following result is also a consequence of the combination 
of \cite[Corollary~3.2.5]{Co} and  \cite[Theorem~8.2.1]{Co}. 
We provide more details in the proof. 
 
\begin{prop} \label{thm::32} 
Suppose that $\mf g$ is of type I. 
Then, for each $\la \in \Sigma^+_\fp$, we have 
$$\epd_{\mc O^\fp}\, T^\fp(\la) = \epd_{\mc O^\fp_\oa}\,  T^\fp_\oa(\la).$$ 	
\end{prop}

\begin{proof} 
Define $\xi_{\pm}\in \h^\ast$ such that 
$\C_{\xi_\pm }\cong \Lambda^{\text{top}} \mf g_{\pm 1}$, as 
$\mf g_\oa$-modules. We note that
\begin{align*}
&\Ind \Delta_\oo^{\fp}(\lambda) \cong \Ind_{\mf g_\oo}^{\mf g} 
\Ind_{\mf p_\oo}^{\mf g_\oo}L_{\mf l}(\lambda) \cong   
\Ind_{\mf p}^{\mf g} \Ind_{\mf p_\oo}^{\mf p}L_{\mf l}(\lambda).
\end{align*} 
This means that $\Ind \Delta_\oo^{\fp}(\lambda) $ has a 
$\Delta^{\fp}$-flag starting at $\Delta^{\fp}(\la+\xi_+)$. 
Therefore, the module $T^\fp(\la  +\xi_+)$
is a direct summand of the module $\Ind T_\oa^\fp(\la)$.
Using Lemma~\ref{lem::31}, we conclude that 
$$\pd\, T^\fp(\la) \leq \pd\, \Ind T^\fp_\oa (\la -\xi_+) 
=\pd \,T^\fp_\oa(\la-\xi_+)=\pd \,T^\fp_\oa(\la).$$
Here the last equality follows by tensoring with 
$\Lambda^{\text{top}} \mf g_{\pm 1}$.
	
Next, we claim that 
$ T_\oa^\fp(\la+\xi_-) \hookrightarrow \Res T^\fp(\la)$. 
To see this, we note that 
\begin{align}
&\Delta^\fp_\oa(\la +\xi_-) \hookrightarrow \Lambda^{\bullet} \mf g_{-1}\otimes \Delta^\fp_\oa(\la) \cong  \Res \Delta^\fp(\la)  \hookrightarrow  \Res T^\fp(\la).
\end{align} 
This implies that there is an indecomposable  direct summand of 
$\Res T^\fp(\la)$ which has a $\Delta^\fp_\oa$-flag starting at 
$\Delta^\fp_\oa(\la+\xi_-)$. In particular, 
$T_\oa^\fp(\la+\xi_-)$ is a direct summand
of $\Res T^\fp(\la)$.
Consequently, we have  
$$\pd \, T_\oa^\fp(\la)= \pd\, T_\oa^\fp(\la+\xi_-) \leq \pd\, \Res T^\fp(\la)\leq \pd\, T^\fp(\la)$$ 
by Lemma \ref{lem::31}. 
This claim of the proposition follows. 
\end{proof}

\subsubsection{Parabolic dimension shift} 

This subsection is devoted to the following formula
for projective dimension.

\begin{thm} \label{thm::510}
Suppose that $M\in \mc O^\fp$ with $\epd_{\mc O^\fp}M<\infty$. Then we have 
\begin{align*}
&\epd_{\mc O^\fp}M = \epd_{\mc O}M   -2\ell(w_0^\fp).
\end{align*} 
\end{thm}

\begin{proof} 
We first claim that 
\[\pd_{\mc O^\fp}M < \infty \Rightarrow \pd_{\mc O}M < \infty. \]
To see this, we let $Q$ be a projective module in $\mc O^\fp$.
The adjoint pair  $\Ind$ and $\Res$ of functors between 
$\mc O_\oa^\fp$ and $\mc O^\fp$ gives an epimorphism 
$\Ind\Res Q \rightarrow Q$, and so $Q$ is a direct summand of 
$\Ind \Res Q$. Using Frobenius reciprocity, it follows that 
\[\text{Ext}_{\mc O}^d(\Ind \Res Q, N)\cong 
\text{Ext}_{\mc O_\oa}^d( \Res Q, \Res N)=0,\] 
for any $d>\pd_{\mc O_\oa}\Res Q$ and $N\in \mc O$.  
As a consequence, $\pd_{\mc O}Q<\infty.$
	
Now suppose that $\pd_{\mc O^\fp}M < \infty$. This means that 
$M$ has a finite projective resolution in $\mc O^\fp$:
\[   0 \rightarrow Q^d\xrightarrow{f^d} Q^{d-1}\xrightarrow{f^{d-1}} 
\cdots \xrightarrow{f^2} Q^1\xrightarrow{f^1} Q^0 \xrightarrow{f^0} M \rightarrow 0.\]
We define $K^s$ to be the kernel of the map $f^s$, for each $s$. Then 
the inequality $\pd_{\mc O}K^{s}<\infty$, together with the exactness 
of $0\rightarrow K^s \rightarrow Q^{s-1} \rightarrow K^{s-1}\rightarrow 0$,
imply that $\pd_{\mc O}K^{s-1}<\infty$. Consequently, 
we have $\pd_{\mc O}M < \infty$.

Now suppose that both $\pd_{\mc O}M$ and $\pd_{\mc O^\fp}M$ are finite. Then, by \cite[Corollary~3.2.5]{Co} and \cite[Theorem~7.2.1(iii)]{Co}, 
we have 
$$\pd_{\mc O}M=\pd_{\mc O_\oa}\Res M
\quad\text{ and }\quad\pd_{\mc O^p}M=\pd_{\mc O^\fp_0}\Res M.$$
The proof now follows from \cite[Theorem~4.1]{CoM2}.
\end{proof}

\begin{cor} Let $M \in \mc O^\fp$ with $\epd_{\mc O^\fp} M<\infty.$
\begin{enumerate}
\item[(1)] 	$M$ is projective in $\mc O^\fp$ if 
and only if  $\epd_{\mc O}M =2\ell(w_0^\fp)$.
\item[(2)] If $\fp=\fg$, then $\epd_{\mc O}M =2\ell(w_0)$.
\end{enumerate}
\end{cor}
\begin{proof} 
Claim~$(1)$ follows directly from Theorem~\ref{thm::510}. 
	
If $\fp=\fg$, then $\mc O^\fp$ is the category of finite dimensional modules. In this case, we have
$\text{fin.dim}\,\mc O^\fp =  2\ell(w_0) - 2\ell(w_0^\fp) =0$ 
by \eqref{eq::egfindim}, which implies that $\pd_{\mc O^\fp}M=0$.
Claim~$(2)$ now follows from Theorem~\ref{thm::510}. 
\end{proof}

\subsection{Injective   dimension of modules in $\mc O^\fp$} \label{sect::54} 

Let $\mu$ be a dominant integral weight whose singularity determines the parabolic 
subalgebra $\fp_\oa$.  Following \cite[Section~1.4]{CCC}, there is a parabolic  
subalgebra $\hat \fp$ (in the sense of \cite[Section~2.4]{Ma}) of $\mf g$ such 
that $\hat \fp_\oa$ is determined by $\hat \mu:= -w_0\mu$. For instance, if $\mf g$ 
is of type I, then $\hat \fp= \hat \fp_\oa \oplus \mf g_{-1}$. 
Also, we have $\hat{\mf b} =\mf b^r$ for $\pn$ from Section \ref{sect::32}. 

We recall, see \cite[Section~1.3]{CCC}, that there is a duality 
(i.e. an anti-equivalence) $\bf D: \mc O^\fp\rightarrow \mc O^{\hat \fp}$ such that 
\begin{align} \label{eq::D}
&{\bf D} L^\fp(\la) =L^{\hat \fp}(-w_0\la),  ~{\bf D}\Delta^\fp(\la) = \nabla^{\hat \fp}(-w_0\la), \text{ and }{\bf D}T^\fp(\la) = T^{\hat \fp}(-w_0\la).
\end{align}
We refer   to \cite[Proposition~3.4, Lemma~3.6]{CCC} for more details. 
Combined with the results in Subsection~\ref{sect::52}, we have the following statement.
 
\begin{prop}\label{prop::injdim} 
For any $M\in \mc O^\fp$, we have $ \eid_{\mc O^\fp}M= \epd_{\mc O^{\hat \fp}}{\bf D}M.$
	
	Furthermore, we have the following properties: 
	\begin{itemize}
		\item[(1)]  Suppose that $\epd_{\mc O^{ \fp}} M<\infty$, then 
		\begin{align}
		&\eid_{\mc O^\fp} M = \eid_{\mc O_\oa^{\fp}} \Res M. \label{eq::551}
		\end{align}

		\item[(2)] Suppose that $\eid_{\mc O^{ \fp}} M<\infty$, then
	\begin{align}
	&\eid_{\mc O^\fp}M = \eid_{\mc O}M   -2\ell(w_0^\fp). \label{eq::552}
	\end{align}

		\item[(3)] Suppose that $\g$ is of type I. Then, for any $\la \in \Sigma^+_\fp$, we have
	\begin{align}
	&\eid_{\mc O^\fp}\, P^\fp(\la) =\eid_{\mc O^\fp_\oa}\, P_\oa^{ \fp}(\la), \label{eq::553}\\ 
	&\eid_{\mc O^\fp}\, T^\fp(\la) = \eid_{\mc O^\fp_\oa} \,T_\oa^{  \fp}(\la).  \label{eq::554}
	\end{align}
		\end{itemize}
\end{prop}

\begin{proof} 
Applying $\bf D$, we have   
	\begin{align}
	&\id_{\mc O^\fp} M =\pd_{\mc O^{\hat \fp}}{\bf D}M.
	\end{align}
Suppose that $\pd_{\mc O^{\hat \fp}}{\bf D}M <\infty$. We note that 
$\Res$ intertwines $\bf D$ and the usual  duality on $\mc O^\fp_\oa$ 
which we also denote by $\bf D$, abusing notation
(cf. the proof of \cite[Theorem~3.7]{CCC}). 
Therefore, from \cite[Corollary~3.2.5]{Co} and \cite[Theorem~7.2.1]{Co},  it follows that 
	$$\pd_{\mc O^{\hat \fp}}{\bf D}M= \pd_{\mc O_\oa^{\hat \fp}}\Res{\bf D}M=\pd_{\mc O_\oa^{\hat \fp}}{\bf D}\Res M=\id_{\mc O_\oa^{\fp}} \Res M.$$ 
	This proves Part $(1)$ and Equation \eqref{eq::551}.  
	 
 To prove Part $(2)$, we note that, from Theorem \ref{thm::510} and  the proof of \cite[Theorem~3]{Ma}, we obtain $\id_{\mc O^\fp} M<\infty \Rightarrow \id_{\mc O}M <\infty$. Therefore we may assume that both $\id_{\mc O} M$ and $\id_{\mc O^\fp}M$ are finite. From Equation \eqref{eq::551} and \cite[Theorem 4.1]{CoM2}, we have 
	\begin{align}
	&\id_{\mc O^\fp} M = \id_{\mc O_\oa^{\fp}} \Res M = 
	\id_{\mc O_\oa} \Res M-2\ell(w_0^\fp) =\id_{\mc O} M -2\ell(w_0^\fp).
	\end{align} This proves Equation \eqref{eq::552} and Part $(2)$.
	An alternative proof of Part $(2)$ is due the following observation:
	\begin{align*}
	&\pd_{\mc O^\fp}M = \pd_{\mc O^{\hat \fp}}{\bf D}M = \pd_{\mc O}{\bf D}M   -2\ell(w_0^{\hat \fp}) \\
	&=  \pd_{\mc O}{\bf D}M   -2\ell(w_0^\fp) = \pd_{\mc O^\fp}{\bf D}M=\id_{\mc O^\fp}M.
	\end{align*} 
	
	 It remains to prove Part $(3)$. 
By \eqref{eq::D}, we have $$
\id_{\mc O^\fp} P^\fp(\la) = \pd_{\mc O^{\hat \fp}} I^{\hat \fp}(-w_0\la),~
\id_{\mc O^\fp} T^\fp(\la) = \pd_{\mc O^{\hat \fp}} T^{\hat \fp}(-w_0\la).$$
 It follows that 
 \[\id_{\mc O^\fp} P^\fp(\la) = \pd_{\mc O_\oa^{\hat \fp}} I_\oa^{\hat \fp}(-w_0\la)
 =\pd_{\mc O_\oa^{\fp}} I_\oa^{ \fp}(\la)=\id_{\mc O_\oa^{\fp}} P_\oa^{ \fp}(\la), \]
  \[ \id_{\mc O^\fp} T^\fp(\la) = \pd_{\mc O_\oa^{\hat \fp}} T^{\hat \fp}_\oa(-w_0\la)= 
  \pd_{\mc O_\oa^{\fp}} T_\oa^{ \fp}(\la), \]
 by  Proposition \ref{prop::53}, Theorem \ref{thm::32} and \cite[Theorem 11]{So90}. 
\end{proof}

\subsection{A sufficient condition for infiniteness of projective dimension}  \label{sect::55}
We recall the  definitions of the {\em associated variety} from \cite{DS}. 
Let $M\in \g$-Mod and $x\in \mf g_{\ov 1}$ with $[x,x]=0$.
Define $$M_x:= \text{ker}_x/xM,$$ where $\text{ker}_x$ is the kernel 
of the map $x: M\rightarrow M$ sending $m$ to $xm$, where $m\in M.$ 
Then the associated variety $X_M$ of $M$ is defined as 
\[X_M:=\{x\in \g_\ob|~[x,x]=0 \text{ and } M_x\neq 0\}.\]

The following statement is \cite[Lemma~2.2(i)]{DS}:

\begin{lem}\label{lem::DS}
Let $M\in \g\emph{-Mod}$ be a summand of $\Ind N$, for some $N\in \mf g_\oa\emph{-Mod}$. Then $X_M=0$.
\end{lem} 

\begin{lem} \label{lem::36}
Let $0\rightarrow A \rightarrow B  \xrightarrow{f}  C  \rightarrow 0$ be a short exact 
sequence in $\g\emph{-Mod}$. Let $x\in \mf g_\ob$ be nonzero with $[x,x]=0$. 
Suppose that $A_x =B_x =0$. Then we have  $C_x =0.$
	
In particular, if $X_A=X_B=0$, then $X_C=0$.
\end{lem}

\begin{proof}
Let $c\in C$ be such that $xc=0$. We want to show that there exists 
$b\in B$ such that $c=xf(b)$, which implies that $C_x =0$. 
To see this, we first pick $d\in B$ such that $f(d)=c$. 
Then $f(xd)=0$ implies that $xd\in A$. 
	
Now $x(xd) =x^2d =0$ (due to $x^2=0$) implies $xd\in  xA$ and so 
there is $a\in A$ such that $xd =xa$ since $A_x =0$. 
Finally, we note that $x(d-a)=0$ implies that $d-a \in xB$ since $B_x=0$. 
Therefore there is $b\in B$ such that $d-a =xb$. We now apply $f$ and obtain 
$$c=f(d)=f(d-a)=xf(b),$$ as desired.
\end{proof}

The following observation is known (see, e.g., \cite[Section 4, proof of Theorem 4.1]{CS} and \cite[Lemma 2.2]{DS}).

\begin{cor} \label{cor::59} 
Suppose $\tilde{\mc C}$ and $\mc C$ are respective  subcategories of $\g$\emph{-Mod} and 
$\g_\oa$\emph{-Mod} satisfying assumptions \eqref{eq::111}-\eqref{eq::113}. 
If $M \in\tilde{\mc C}$ is such that $X_M\neq 0$, then $\epd_{\tilde{\mc C}} M=\infty$. 	
			
In particular,	if $M \in \mc O^\fp$ is such that $X_M\neq 0$, 
then $\epd_{\mc O^\fp} M=\infty$.
\end{cor}

\begin{proof}
Suppose that there exists a finite projective resolution of $M$ in $\tilde{\mc C}$:
\[0 \xrightarrow{f^{d+1}} P^d \xrightarrow{f^d} \cdots \xrightarrow{} P^2\xrightarrow{f^2}  P^1\xrightarrow{f^1}  P^0 \rightarrow M \xrightarrow{f^0}  0.\]
	
Let $0\leq s\leq d$. By \cite[Proposition~2.2.1]{Co}, every projective module $P^s$ is a direct summand of a module induced from $\mc C$ to $\tilde{\mc C}$. By Lemma \ref{lem::DS} we 
thus have $X_{P^s} = 0$. 
	 
Define $K^s$ to be the kernel of $f^s$. For any $s$, we have an exact sequences 
$$0\rightarrow K^s \rightarrow P^{s-1} \rightarrow K^{s-1}\rightarrow 0.$$ 
In particular, the fact that $K^d=P^d$ implies that $K^{d-1}_x=0$ by Lemma \ref{lem::36}.  By induction on $s$, we may conclude that $M_x=K^0_x =0$.  The claim of the lemma follows. 
\end{proof}

\begin{ex}
If $\mf g_{\ov 1}$ contains a non-zero $x$ such that $[x,x]=0$, 
then it is obvious (see also \cite[Lemma 2.2]{DS}) 
that the associated variety $X_\C$ of  the trivial module $\C$ is non-zero. 
Therefore we may conclude that both $\pd_{\g\text{-Mod}}\C$ and $\pd_{\mc O}\C$ are infinite.
\end{ex}

\begin{rem}
It is worth pointing out that the condition $X_M=0$ is independent of the choice of the parabolic subalgebra $\fp$  (compare with Theorem \ref{thm::510}). It is natural to ask the following question:
	
{\bf Question:} Let $M\in \mc O$. Is it true that $\pd_{\mc O} M<\infty \Leftrightarrow X_M=0?$
	 
This question has the affirmative answer for the general linear Lie superalgebra 
$\gl(m|n)$ by \cite[Theorem 4.1]{CS}.
\end{rem}

The following lemma addresses infiniteness of projective dimension 
for $\fp$-highest weight modules.

\begin{cor} \label{coro::infproj} 
Suppose that there is a weight $\alpha$ of $\mf g_1$ such that $-\alpha$ is not a weight of $\Lambda^{\bullet} \mf g_{-1}\otimes U(\mf n^-_\oa).$ Then the associated variety
of every non-zero quotient of $\Delta^\fp(\la)$ is non-zero.
Consequently, every such quotient has infinite projective dimension in both $\mc O$ and 
$\mc O^\fp$, for any $\la \in \Sigma_\fp^+$.
	
In particular, every block of $\mc O^\fp$ contains infinitely many simple objects.
\end{cor}

\begin{proof} 
Assume that $M$ is a non-zero quotient of $\Delta^\fp(\la)$
of finite projective dimension. 
Let $X_{\alpha} \in \g_1$ be a root vector of  root $\alpha$. 
	
Pick a non-zero highest weight vector $v\in M$.   By Corollary \ref{cor::59}, we have $M_{X_{\alpha}}=0$. Therefore $X_{\alpha} v=0$ implies that  there is $m\in M$ such that $v=X_{\alpha}m$.
However,  $\la -\alpha$ is not a weight of $M$, a contradiction. 
	
Suppose that $\mc O^\fp_\la$ is a block of $\mc O^\fp$ containing only finitely many simples. 
Combining existence of tilting modules  \cite[Section 4.3]{Ma} 
(also, see \cite[Theorem 3.5]{CCC}), the Ringel duality \cite[Section 5.4]{Ma} 
(also, see, \cite[Corollary 3.8]{CCC}) and  the BGG reciprocity  \cite[Lemma 3.3]{CCC} 
for $\mc O^\fp$, we may conclude that $\mc O^\fp_\la$ contains a simple tilting module. 
This is a contradiction to Proposition \ref{thm::32}. This completes the proof. 
\end{proof}
 
\begin{ex}
Let $\mf k$ be classical of type I such that there exists $d\in \mc Z(\mf k)$ and 
$d$ acts on $\mf k_{1}$ as a non-zero scalar. Set $\mf g:=\mf g_\oa\oplus \mf g_{1}$. 
Then every simple module in the category $\mc O$ for $\mf g$ has infinite projective dimension. 
\end{ex} 

We are going to apply Corollary \ref{coro::infproj} to the periplectic Lie 
superalgebra $\pn$ later in Proposition \ref{prop::512}.

\section{Examples: projective dimensions of structural supermodules over $\pn$ and  $\mf{osp}(2|2n)$} \label{Sect::6}
	
\subsection{Example I: projective dimension of  modules for $\mf{osp}(2|2n)$} \label{Sect::projdimosp}

\subsubsection{}
The orthosymplectic Lie superalgebra $\mf{osp}(2|2n)$ is the 
following subsuperalgebra of $\mf{gl}(2|2n)$:
\begin{equation}
\mf{osp}(2\vert 2n)=
\left\{ \left( \begin{array}{cccc} c &0 & x &y\\
0 & -c& v & u\\
-u^t& -y^t & a &b\\
v^t &x^t & c& -a^t \\
\end{array} \right):
\begin{array}{c}
c\in \C;\,\, x,y,v,u\in \C^{n};\\
a,b,c\in \C^{n^2};\\
b=b^t,\,\, c=c^t.
\end{array}
\right\}. \label{eq::osp}
\end{equation}
The even part $\mf{osp}(2|2n)_\oa$ of $\mf{osp}(2|2n)$ is isomorphic to 
$\C \oplus \mf{sp}(n)$, and the odd part is isomorphic to $\C^2\otimes \C^{2n}$. 
We refer the reader to \cite[Section 1.1.3]{ChWa12} and \cite[Section 2.3]{Mu12} 
for more details.  Recall from  Section \ref{sect::321} that we denote by 
$E_{ab}$ the elementary matrices. We now define 
$$H_\vare:=E_{11}-E_{22},~H_i:=E_{2+i,2+i}-E_{2+i+n,2+i+n},
\quad \text{ for }1\leq i\leq n.
$$
The Cartan subalgebra $\mf h$ of $\mf{osp}(2|2n)$ 
is spanned by $H_\vare$ and $H_1,H_2,~\ldots, H_n$. We let 
$\{\vare,\delta_1,\delta_2,\ldots,\delta_n\}$ be 
the basis of $\mf h^\ast$ dual to  $\{H_\vare,H_1,H_2,\ldots, H_n\}$.
The  bilinear form $({}_-,{}_-):\mf h^\ast \times \mf h^\ast \rightarrow \C$ is given by $(\vare,\vare) =1$ and $(\delta_i,\delta_j) =-\delta_{i,j}$, for $1\leq i,j\leq n.$
 
We define the set $\Phi^+_\oa$ of even positive roots  and 
the set $\Phi^+_\ob$ of odd positive roots  by 
\begin{align}
&\Phi^+_\oa:=\{\delta_i \pm\delta_j, 2\delta_p|~1\leq i< j <n, ~1\leq p \leq n \}.\\
&\Phi_\ob^+:=\{\vare \pm \delta_p|~1\leq p\leq n\}.
\end{align} 
This gives rise to a triangular decomposition 
$\mf g=\mf n^+\oplus \mf h \oplus \mf n^-$ with Borel subalgebra $\mf b:= \mf b_\oa \oplus \mf g_1.$

\subsubsection{Simple supermodules over $\mf{osp}(2|2n)$}
The aim of this subsection is to study the projective dimension of simple 
supermodules over arbitrary basic classical Lie superalgebras of type I including $\mf{osp}(2|2n)$.
  
We recall the notation of typical weights for basic classical Lie superalgebras 
of type I, see, e.g., \cite[Section 2.4]{Gor2}. 
Let $X^{-}\in \Lambda^{\text{max}}\mf g_{-1}$ and 
$\quad X^{+}\in \Lambda^{\text{max}}\mf g_{1}$ be non-zero vectors. 
Then by   \cite[Lemma 6.5]{CM} (also, see, \cite[Section 4]{Gor1}) it follows that  
\begin{align}
&X^{+} X^- = \Omega + \sum_i x_i r_i y_i, \label{eq::XYeq}
\end{align}
for some $x_i \in \Lambda(\mf g_{ -1})\backslash \mathbb{C}$, 
$y_i \in \Lambda(\mf g_{1})\backslash \mathbb{C}$, 
$r_i\in U(\mf g_{\oo})$ and $\Omega \in Z(\mf g_\oo)$. 
A weight $\la \in \h^\ast$ is typical if its evaluation at the 
Harish-Chandra projection of $\Omega$ is non-zero.   
It was shown in \cite[Theorem 1]{Ka2} that the typicality of   
$\la$ is equivalent to the simplicity of the Kac module of 
highest weight $\la$. A more general definition of {\em typical Kac modules} 
can be found in \cite[Section 6.2]{CM}.
   
The following proposition shows that all simple supermodules of atypical 
blocks over basic classical Lie superalgebras have infinite projective dimension 
in $\mc O$. This generalizes \cite[Proposition~5.14]{CS} where the case
of the superalgebra 
$\gl(m|n)$ was considered.
	
\begin{prop} \label{thm::basicgpdsimple}
Let $\mf g$ be a basic classical Lie superalgebra of type I. 
Suppose that $\la$ is atypical. Then the associated variety 
$X_{L(\la)}$ is non-zero. In particular, if  $L(\la) \in \mc O^\fp$,
for some $\fp$, then $\epd_{\mc O^\fp}L(\la)=\infty.$
\end{prop}
	
Before we prove Proposition \ref{thm::basicgpdsimple}, we need to 
recall the notion of odd reflection for basic classical Lie superalgebras, 
see, e.g., \cite[Lemma 1]{PS89}, \cite[Section 1.4]{ChWa12} and 
\cite[Section 3]{Mu12}, for more details.
	
Since $\mf g$ is contragradient,  the set of roots of $\mf g_{-1}$ can be obtained 
from the set of roots of $\mf g_{1}$ by multiplying with the scalar  $-1$. 
Set $d:=\text{dim}\mf g_1=\text{dim}\g_{-1}$. 
By \cite[Lemma 1.30, Proposition 1.32]{ChWa12} (also, see \cite[Theorem 3.13]{Mu12}),
the positive root system given by the Borel subalgebra $\mf b_\oa\oplus\mf g_{-1}$ 
can be obtained from that of  $\mf b$ via a sequence of odd reflections. 
Namely, we can fix an ordering of   positive odd roots 
\begin{align}
&\{\alpha_1, \alpha_2, \ldots, \alpha_d\} = \Phi^+_\ob \label{eq::oddfleordering}
\end{align} 
such that the following conditions are satisfied: 
\begin{itemize}
\item[(1)] For any $1\leq i\leq d$, the set  
$R^{i} =\Phi^+_\oa \cup \{-\alpha_1,\ldots ,-\alpha_{i}\} 
\cup  \{\alpha_{i+1},\ldots, \alpha_d\}$ forms a positive root system. 
\item[(2)]  We set $R^{-1}:=R^0$. Then $\alpha_i$ is a 
simple root in $R^{i-1}$ and $-\alpha_i$ is a simple 
root in $R_{i}$, for any $1\leq i\leq d$.
\end{itemize}
		
Let us fix root vectors $X_i\in \mf g_1\cap \mf g^{\alpha_i}$ and 
$Y_i \in \mf g_{-1}\cap \mf g^{-\alpha_i}$, for $1\leq i \leq n$. 
Also,	let $\mf b^{(i)}$ denote the Borel subalgebra corresponding 
to $R^i$, namely, $\mf b^{(i)}_\oa = \mf b_\oa$ and $\mf b^{(i)}_\ob$ 
is generated by all root vectors of odd roots in $R^i$.  
By \cite[Lemma 1.40]{ChWa12}, we have the following lemma: 

\begin{lem} \label{lem::29}
Suppose that $L\in \mc O$ is a simple supermodule having $\mf b^{(i)}$-highest 
weight  $\mu\in \h^\ast$. Let $v\in L$ be a $\mf b^{(i)}$-highest weight vector. 
\begin{itemize}
\item[(1)] 	If $Y_{i+1} v=0$, then $v$ is a  $\mf b^{(i+1)}$-highest weight vector of $L.$
\item[(2)]  If  $Y_{i+1} v\neq 0$, then $Y_{i+1} v$  is a $\mf b^{(i+1)}$-highest 
weight vector of $L$.
\end{itemize}
\end{lem}
	
\begin{ex}
Let $\mf g= \mf{osp}(2|2n)$. We start with the simple system 
(see \cite[Section 1.3.4]{ChWa12})
\[\{\vare-\delta_1,~\delta_2-\delta_3,~\ldots,~\delta_{n-1}-\delta_n,~ 2\delta_n\}.\]
Then, by a direct computation using \cite[Equation (1.44) of Lemma 1.30]{ChWa12}, 
the ordering  \eqref{eq::oddfleordering} in this case can be chosen as follows: 
\begin{align}
&\vare -\delta_1,~\vare-\delta_2,~\ldots,~ \vare- 
\delta_n,~\vare +\delta_n,~\vare +\delta_{n-1},~\ldots,\vare+\delta_1. \label{eq::osproots}
\end{align}
\end{ex}
	
Now we are in a position to  prove Proposition \ref{thm::basicgpdsimple}.

\begin{proof}[Proof of Proposition \ref{thm::basicgpdsimple}] 
Let $v\in L(\la)$ and $v'\in K(\la)$  be highest weight vectors.  
Since $\la$ is atypical, the kernel of the natural epimorphism 
$\phi: K(\la)\rightarrow L(\la)$ is non-zero. 
By \cite[Lemma 3.2]{CM}, the kernel of $\phi$ contains $Y_dY_{d-1}\cdots Y_1v'$ 
and so $Y_dY_{d-1}\cdots Y_1v= 0$. This means that there is 
$0\leq \ell < d$ such that $Y_{\ell} Y_{\ell - 1}\cdots Y_1 v\neq 0$ 
and $Y_{\ell+1} Y_{\ell}\cdots Y_1 v =0$. By 
Lemma \ref{lem::29}, $Y_{\ell} Y_{\ell - 1}\cdots Y_1 v$ is a 
highest weight vector of $L$ with respect to both $\mf b^{\ell}$ 
and $\mf b^{\ell+1}$. Therefore the weight spaces corresponding to the weights 
$\la-\alpha_1 +\cdots -\alpha_{\ell}\pm \alpha_{\ell+1}$ of $L(\la)$ 
are zero. In particular, the vector 
$Y_{\ell} Y_{\ell - 1}\cdots Y_1 v\not \in Y_{\ell+1} L(\la)$. 
This means that that $L(\la)_{Y_{\ell+1}}\neq 0$ 
and the claim  of our lemma follows from Corollary \ref{cor::59}. 
\end{proof}
	
\subsubsection{(Dual) Verma supermodules  over $\mf{osp}(2|2n)$} 

In this subsection, we assume that $\mf g=\mf{osp}(2|2n)$, for some positive integer $n$. 
The following proposition shows that atypical Verma supermodules have infinite projective dimension.  We recall that $\Sigma^+_\fp$ denotes the set of highest weights of simple modules in $\mc O^\fp$.

\begin{prop}  \label{thm::pdospVerma}
Suppose that $\la\in \Sigma_\fp^+$. 
Then the following conditions are equivalent:
\begin{itemize}
\item[(1)] $\epd_{\mc  O^\fp}\Delta^\fp(\la) <\infty$.
\item[(2)] $\epd_{\mc  O}\Delta^\fp(\la) <\infty$.
\item[(3)] $\epd_{\mc  O^\fp}\Delta^\fp(\la) = \epd_{\mc  O^\fp_\oa}\Delta^\fp(\la) 
= \epd_{\mc  O}\Delta^\fp(\la) -2\ell(w_0^\fp)$.
\item[(4)] $\la$ is typical.	
\end{itemize} 
\end{prop}

\begin{proof} 
We first note that the implication $(1)\Rightarrow (3)$ follows from  
\cite[Corollary~3.2.5]{Co} and \cite[Lemma 6.3.2]{Co}. 
Therefore we have $(1)\Leftrightarrow (3)$. From \cite[Theorem~1.3.1, Theorem 1.4.1]{Gor2} 
it follows that $(4)\Rightarrow (1), (2)$.  Therefore it remains to show that, if 
$\la$ is atypical, then the associated variety of $\Delta^\fp(\la)$ is non-zero. 
	
We first define an ordering of roots of $\mf g_{-1}$ 
(i.e. roots in $\Phi_\ob \backslash \Phi^+_\ob$) as follows:
\begin{align}
&\alpha_1:=-\vare -\delta_1<\alpha_2:=-\vare-\delta_2<\ldots< \alpha_{n}:=-\vare- \delta_n\\&< \alpha_{n+1}:=-\vare +\delta_n
<\alpha_{n+2}:=-\vare +\delta_{n-1}<\ldots<\alpha_{2n}:=-\vare+\delta_1. \label{eq::osproots2}
\end{align} 
By \eqref{eq::osp}, for each $1\leq i \leq n$, there exist non-zero 
vectors $X_i \in \g^{-\alpha_i}$ and $Y_{i}\in \g^{\alpha_i}$ such that 
\begin{align}
&[X_i,Y_i] = H_\vare-H_{i},\text{ for $1\leq i\leq n$},\\
&[X_{n+i},Y_{n+i}] = -H_\vare-H_{n-i+1} ,\text{ for $1\leq i\leq n$}.
\end{align} 
We observe the following crucial fact: 
\begin{align}
&\mf n^+_\oa Y_i \subset \bigoplus_{i< j}\C Y_j,~[X_i,Y_j] \subset \mf n^+_\oa,
\label{eq::611} 
\end{align} 
for any $1\leq i<j\leq 2n.$  As a consequence, we have  
\begin{align}
&X_iY_{j}\cdots Y_{2n}v  =0, \label{eq::613}
\end{align} for any $1\leq i<j\leq 2n$. 

Let $\displaystyle\rho : =-n \vare+ \sum_{1\leq i\leq n} (n-i+1) \vare_i$ be the Weyl 
vector, i.e., $\rho$ is the half of the sum of roots in $\Phi_\oa^+$ minus the half of
the sum of roots in $\Phi^+_\ob$. Since $\la$ is atypical, we have  
\begin{displaymath}
\prod_{\alpha\in\Phi^+_{\bar{1}}}(\la+\rho,\alpha) =0,
\end{displaymath} 
by \cite[Subsection~4.2]{Gor1}. Therefore there exists $\alpha = \vare+c\delta_i$ 
such that $(\la +\rho ,\alpha) =0$ and $c=\pm 1.$ We let $v$ be a highest 
weight vector of $\Delta^\fp(\la)$. Also, we consider the element 
$\displaystyle\la = \la_\vare \vare+\sum_{1\leq i\leq n} \la_i \delta_i$. 
	
Suppose that $c=1$. By \eqref{eq::613} and direct computation, we 
get $X_iY_{i+1}\cdots Y_{2n}v =0$ and 
\begin{align*}
&X_iY_iY_{i+1}\cdots Y_{2n}v \\
&= (H_\vare -H_i)Y_{i+1}\cdots Y_{2n}v \\
&= (\la_\vare -\la_i-2n+i -1) Y_{i+1}\cdots Y_{2n}v \\
&=(\la+\rho, \alpha)Y_{i+1}\cdots Y_{2n}v =0.
\end{align*}
Since the weight subspace of $\Delta^\fp(\la)$ of weight  
$\displaystyle\la-\sum_{p\leq q\leq 2n}\alpha_q$ is one-dimensional and spanned by 
$Y_{p}Y_{p+1}\cdots Y_{2n}v$, for any $1\leq p\leq 2n$, we may conclude 
that $Y_{i+1}\cdots Y_{2n}v \notin X_i\Delta^\fp(\la)$, for otherwise 
$Y_{i+1}\cdots Y_{2n}v=0$, which is a contradiction. Consequently, 
$X_{\Delta^\fp(\la)}\neq 0$. Hence, in this case, the fact that
the associated variety of $\Delta^\fp(\la)$ is non-zero 
follows from  Corollary~\ref{cor::59}. 
   
Suppose that $c=-1$. Set $\ov i:=n-i+1$, for each $1\leq i\leq n$. 
Similarly, using \eqref{eq::613} and direct computation, we 
obtain $X_{n+\ov i}Y_{n+\ov i +1 +1}Y_{2n}v=0$ and 
\begin{align*}
&X_{n+\ov i}Y_{n+\ov i}Y_{n+\ov i+1}\cdots Y_{2n}v \\
&= -(H_\vare +H_i)Y_{n+\ov i+1}\cdots Y_{2n}v \\
&=(-1) (\la_\vare +\la_i -i+1) Y_{n+\ov i+1}\cdots Y_{2n}v \\
&=-(\la+\rho, \alpha)Y_{n+\ov i+1}\cdots Y_{2n}v =0.
\end{align*} 
As above, this implies that $Y_{n+\ov i+1}\cdots Y_{2n}v \notin X_{n+\ov i}\Delta\fp(\la)$ 
and yields  $X_{\Delta^\fp(\la)}\neq 0$. Again, the fact that
the associated variety of $\Delta^\fp(\la)$ is non-zero
now follows from  Corollary \ref{cor::59}.  
This completes the proof. 
\end{proof}

\begin{cor}
Suppose that $\la\in \Sigma_\fp^+$. 
Then the following are equivalent:
\begin{itemize} 
		\item[(1)] $\epd_{\mc  O^\fp}\nabla^\fp(\la) <\infty$.
\item[(2)] $\epd_{\mc  O}\nabla^\fp(\la) <\infty$.
\item[(3)] $\epd_{\mc  O^\fp}\nabla^\fp(\la) = \epd_{\mc  O^\fp_\oa}\nabla^\fp(\la) = \epd_{\mc  O}\nabla^\fp(\la) -2\ell(w_0^\fp)$.
\item[(4)] $\la$ is typical.	
\end{itemize} 
\end{cor}
\begin{proof}

The implication $(1)\Rightarrow (3)$ follows from  
\cite[Corollary~3.2.5]{Co} and \cite[Lemma 6.3.2]{Co}. Therefore the statements in  $(1)$ and $(3)$ are equivalent. 
	
	We recall that $\mf g= \mf{osp}(2|2n)$ is  one of the contragredient Lie superalgebras in \cite[Theorem 5.1.5]{Mu12}. Therefore $\mc O^\fp$ admits a simple-preserving duality (see, e.g., \cite[Section 13.7]{Mu12}) which sends dual Verma supermodules to Verma supermodules. By \cite[Theorem 3]{Ma}, it follows from Proposition \ref{thm::pdospVerma} that the statements in $(1),(2)$ and $(4)$ are equivalent.

	\end{proof}

\subsection{Example II: projective dimension of  modules for $\pn$} \label{Sect::projdimpn}

In this section, we assume that $\mf g=\pn$, for some positive integer $n$. 
It is known from \cite[Section~ 5.5]{CCC} that  every parabolic category 
is equivalent to the parabolic category given 
by a parabolic subalgebra  $\mf p$ with $\mf p_\ob=\mf  g_1$. 
We recall that  $\Sigma^+_\fp$ denotes the set of highest weights of simple modules
in $\mc O^\fp$.

\begin{prop} \label{prop::512}
Suppose that $M\in \mc O^\fp$ is a non-zero quotient of $\Delta^\fp(\la)$, 
for some $\la \in \Sigma^+_\fp$. 
Then the associated variety $X_M$ is non-zero. In particular, 
$$\epd_{\mc O}M =\epd_{\mc O^\fp}M =\infty.$$ 
In particular, the modules $L(\la),~K(\la)$ and $\Delta^\fp(\la)$ 
all have infinite projective dimension in both $\mc O$ and $\mc O^\fp$. 
\end{prop}

\begin{proof}
We may observe that the element $2\vare_n$ is a root of $\mf g_1$ but $-2\vare_n$ 
is not a weight of $\Lambda^{\bullet} \mf g_{-1}\otimes U(\mf n^-_\oa)$. 
The assertion of the proposition now follows directly from Corollary~\ref{coro::infproj}. 
\end{proof}

We recall from \cite{Se02} that a weight 
$\displaystyle\la =\sum_{1\leq i\leq n}\la_i\vare_i\in \h^\ast$ is called {\em atypical}  if  
\[  \prod_{1\leq i\neq j \leq  n}(\la_i -\la_j +j-i-1) =0. \]
We describe the projective dimension of costandard objects in the  following proposition. 

\begin{thm}
Let  $\la \in \Sigma^+_\fp$.  Then the following
conditions are equivalent:
\begin{enumerate}
\item $\epd_{\mc O^\fp} \nabla^\fp(\la) <\infty.$
\item $\epd_{\mc O^\fp} \nabla^\fp(\la) =\epd_{\mc O_\oa^\fp} \nabla^\fp_\oa(\la).$
\item $\la$ is typical.
\end{enumerate}
\end{thm}

\begin{proof}
From \cite[Corollary 3.2.5]{Co} and \cite[Lemma 6.3.2]{Co}  it follows that 
$$\pd_{\mc O^\fp}\nb^\fp_{\mc O^\fp}(\la) = \pd_{\mc O^\fp_\oa}\nb^\fp_\oa(\la), $$ 
if $\pd_{\mc O^\fp}\nabla^\fp(\la)<\infty.$ Therefore it remains to show that 
$\pd_{\mc O^\fp}\nb^\fp(\la)<\infty$ if and only if $\la$ is typical.
	
We first suppose that $\la$ is typical. Then, by \cite[Theorem 4.6]{CP}, we have 
\[(T^\fp(\la): \nabla^\fp(\mu)) = (T^\fp_\oa(\la): \nabla^\fp_\oa(\mu)),\]
for any  $\mu\in \Sigma_\fp^+$. 
We start by fixing $\la$ such that $T^\fp(\la) = \nabla^\fp(\la)$. That is,  
let $J^\fp_\la$ be the  intersection of the set of shortest representatives 
in $W^\fp \backslash W$ and the set of longest representatives in $ W/W_\la$.
We choose $\la$ such that it is  a  minimal element in $J^\fp_\la\cdot \la$.
We want to show that $\pd_{\mc O^\fp} \nabla(\mu) <\infty$, 
for any $\mu \in J^\fp_\la \cdot \la$.  
	 	 
We first note that 
\begin{align}
&\pd_{\mc O^\fp}\nabla^\fp(\la) =\pd_{\mc O^\fp} T^\fp(\la) = 
\pd_{\mc O_\oa^\fp} T^\fp_\oa(\la) = \pd_{\mc O_\oa^\fp} \nabla^\fp_\oa(\la), \label{eq::55}
\end{align} 
by Theorem \ref{thm::32}. 
Assume that $x\in J^\fp_\la$ is such that $\pd_{\mc O^\fp}\nabla(y\cdot \la)<\infty,$ 
for any $y\in J^\fp_\la$ with $x\cdot \la> y\cdot \la.$ By \cite[Theorem 3.5]{CCC}, 
there is a submodule $K$ of $T^\fp(x\cdot \la)$ such that we have the short exact 
sequence
\[0\rightarrow K\rightarrow T^\fp(x\cdot \la) \rightarrow \nabla^\fp(x\cdot \la) \rightarrow  0,\] where $(K:\nabla^\fp(z\cdot \la))>0$    only if $z\in J^\fp_\la$  and  $x\cdot \la >z\cdot \la$. 
From our assumption, we have $\pd K<\infty$. 
Since $\pd_{\mc O^\fp} T^\fp(x\cdot \la)<\infty$, it  follows
that $\pd_{\mc O^\fp} \nabla^\fp(x\cdot \la)<\infty$.  
This completes the proof of the implication $(3)\Rightarrow (1).$
	
Before  proving the direction  $(1)\Rightarrow (3)$, we make the following observations. 
	
For any positive odd root $\alpha = \vare_i +\vare_j$, we let $X_\alpha\in \g_1$ be a 
root vectors for $\alpha$. If $\alpha =\vare_i +\vare_j$ with $i< j$, we can pick root 
vectors  $Y_{\alpha}\in \g_{-1}$ for odd roots $-\alpha$  such that 
$[Y_{\alpha}, X_{\alpha}] =H_i-H_j$. 
	
Define a total ordering $\preceq$ on the set $\Phi_\ob^+:=\{\vare_i+\vare_j|~1\leq i<j\leq n\}$ by letting $$\vare_i +\vare_{j+1}\prec \vare_i+\vare_{j},~2\vare_{i}\prec \vare_{i-1} +\vare_n,$$
for $1< i\leq n$ and $1\leq j<n$. For any $1\leq i<j\leq n$ and any subset $$I\subset \hat I:=\{\vare_i+\vare_t|~i\leq t\leq j-1\},$$ we define 
\begin{align*}
&S_{i,j,I}:=\{\vare_s+\vare_t|~s<t,~1\leq s\leq i-1,~1\leq t\leq j\}\cup I.
\end{align*}
	 
Let $\la\in \h^\ast$ and  $v_\la \in \nabla^\fp_\oa(\la)\subset \nb^\fp(\la)$  
be a highest weight vector. Set 
$$X^{i,j,I}:=\prod_{\alpha \in S_{i,j,I}}X_\alpha\otimes v_\la \in \nb^\fp(\la).$$ 
We observe that 
\begin{enumerate}[(a)]
\item 	For any $\alpha \in \Phi_\ob^+,$ we have  
$$[X_\alpha,\mf n_\oa^+] \subseteq  \bigoplus_{\beta\in\Phi_\ob^+, ~\alpha \prec \beta}\C X_\beta.$$  \label{eq::1}
\item $[Y_{\vare_i+\vare_j}, X_{\vare_k+\vare_i}]\in \mf n_\oa^+$ 
has weight $\vare_k-\vare_j$, for $1\leq k\leq i$.   
\item $[Y_{\vare_i+\vare_j}, X_{\vare_i+\vare_k}]\in \mf n_\oa^+$ 
has weight $\vare_k-\vare_j$, for $i<k<j$. 
\item $[Y_{\vare_i+\vare_j}, X_{\vare_k+\vare_j}]\in \mf n_\oa^+$ 
has weight $\vare_k-\vare_i$, for $1\leq k<i$.  \label{eq::4}
\end{enumerate}   
As a consequence, from \eqref{eq::1}-\eqref{eq::4} above, it follows that 
\begin{align}
&Y_{\vare_i +\vare_j} X^{i,j,I}=0. \label{eq::k}
\end{align}

We can now proceed with the proof of $(1)\Rightarrow (3)$. Suppose that $\la$ is 
integral and atypical with $\la_i -\la_j +j-i=\pm 1$, for some $i<j$, satisfying 
$\pd\nb^\fp(\la)<\infty$. Then, by Corollary~\ref{cor::59}, the associated variety 
$X_{\nabla^\fp(\la)}$ is zero.  
   
We have to consider the following two cases:
       \vskip 0.2cm
       
{\bf Case 1}.
Suppose that $\la_i -\la_j +j-i+1=0$.
    
Let $x:=X^{i,j,\hat I}$ and $y:=Y_{\vare_i+\vare_j}$. 
Then   $X_{\nabla^\fp(\la)}=0$ implies that $\nb^\fp(\la)_y=0$. 
Therefore  $x\in y\nb^\fp(\la)$. 
    
Assume that $m\in \nb^\fp(\la)$ is a weight vector such that 
$yx=0$ and $x=ym$. Then $m \in  \C X_{\vare_i+\vare_j}X^{i,j,\hat I}.$ 
This implies that 
\[x=ym\in \C Y_{\vare_i+\vare_j}X_{\vare_i+\vare_j}X^{i,j,\hat I} = 
\C(H_i -H_j)X^{i,j,\hat I} = \la_i -\la_j +j-i+1 =0,  \]
a contradiction. 
\vskip 0.2cm

{\bf Case 2}.
Suppose that $\la_i -\la_j +j-i-1=0$.
Set $I:= \hat I \backslash \{2\vare_i\}$.
    
Let $x:=X^{i,j, I}$ and $y:=Y_{\vare_i+\vare_j}$. 
Then     $X_{\nabla^\fp(\la)}=0$ implies that $\nb^\fp(\la)_y=0$. 
Again, we have  $x\in y\nb^\fp(\la)$. 
      
Assume that $m\in \nb^\fp(\la)$ is a weight vector such that $yx=0$ and $x=ym$. 
Observe that 
$$(H_i-H_j)m = (\la_i-\la_j +j-i-1)m=0,$$  
and 
$$m =   X^{i,j,\hat I}\otimes m_j+ \sum_{i\leq t< j}  
X^{i,j,\hat I\backslash \{\vare_i+\vare_t\}\cup \{\vare_i+\vare_j\}}\otimes m_t,$$
for some $m_j, m_t \in \nb_\oa(\la).$ 
By \eqref{eq::1}-\eqref{eq::4}, for $i\leq t <j$, we have 
$$yX^{i,j,\hat I}\otimes m_j =0,$$
$$y X^{i,j,\hat I\backslash \{\vare_i+\vare_t\}\cup \{\vare_i+\vare_j\}}\otimes m_t  
=(H_i-H_j)X^{i,j,\hat I\backslash \{\vare_i+\vare_t\}}\otimes m_t=0.$$ 
The obtained contradiction completes the proof.
\end{proof}

\vspace{2mm}

\noindent
CC:~Department of Mathematics, National Central University, Zhongli District, Taoyuan City, Taiwan;
E-mail: {\tt cwchen@math.ncu.edu.tw}
\hspace{2cm} 

\noindent
VM:~Department of Mathematics, Uppsala University, Box 480, SE-75106, Uppsala, SWEDEN;
E-mail: {\tt  mazor@math.uu.se}



\end{document}